\documentclass[11pt]{article}
    
    \usepackage[margin=1.0in]{geometry}
    \usepackage{mathtools}
    \usepackage{amsthm}
    \usepackage{fancyhdr}
    \usepackage{hyperref} 
    \usepackage{amssymb}
    \usepackage[capitalise]{cleveref}
    \usepackage{subcaption}
    \usepackage{cmbright}
    \usepackage{anyfontsize}
    \usepackage{xfrac}
    
    \newtheorem{theorem}{Theorem}[section]
    \newtheorem{corollary}{Corollary}[theorem]
    \newtheorem{lemma}[theorem]{Lemma}
    \newtheorem{defn}[theorem]{Definition}
    \newtheorem{remark}[theorem]{Remark}
    \newtheorem{prop}[theorem]{Proposition}
    \crefname{prop}{Proposition}{Propositions}
    
    \newcommand{\bbR}{\mathbb{R}}
    
    \newcommand{\diff}{\,\textnormal{d}}
    \newcommand{\dt}{\Delta t}
    \newcommand{\dv}{\,\diff\bfv}
    \newcommand{\intd}{{\int_{\bbR^d}}}

    \newcommand{\quand}{\quad \mbox{and} \quad}
    \newcommand{\qquand}{\qquad \mbox{and} \qquad}
    \newcommand{\qwhere}{\quad \mbox{where} \quad}
    \newcommand{\qqwhere}{\qquad \mbox{where} \qquad}
    \newcommand{\qquiff}{\qquad \iff \qquad}

    \newcommand{\vertiii}[1]{\left\vert\kern-0.25ex\left\vert\kern-0.25ex\left\vert #1 
        \right\vert\kern-0.25ex\right\vert\kern-0.25ex\right\vert}

    \newcommand{\froNorm}[1]{\| #1 \|_{\textnormal{F}}}

    \newcommand{\Delp}[1]{\Delta^{\ell} #1}
    \newcommand{\Delm}[1]{\Delta^{\ell-1} #1}
    
    \newcommand{\bfa}{\boldsymbol{a}}
    \newcommand{\bfb}{\boldsymbol{b}}
    \newcommand{\bfc}{\boldsymbol{c}}
    \newcommand{\bfk}{\boldsymbol{k}}
    \newcommand{\bfs}{\boldsymbol{s}}
    \newcommand{\bfu}{\boldsymbol{u}}
    \newcommand{\bfv}{\boldsymbol{v}}
    \newcommand{\bfw}{\boldsymbol{w}}
    \newcommand{\bfx}{\boldsymbol{x}}
    \newcommand{\bfy}{\boldsymbol{y}}
    \newcommand{\bfeta}{\boldsymbol{\eta}}
    
    \newcommand{\bfA}{\boldsymbol{A}}
    \newcommand{\bfT}{\boldsymbol{T}}
    \newcommand{\bfU}{\boldsymbol{U}}
    \newcommand{\bfW}{\boldsymbol{W}}
    \newcommand{\bfX}{\boldsymbol{X}}
    \newcommand{\bfY}{\boldsymbol{Y}}

    \newcommand{\veps}{\varepsilon}
    \newcommand{\what}{\widehat}
    
    \allowdisplaybreaks

    \date{27 April 2024}

\title{Implicit Update of the Moment Equations for a Multi-Species, Homogeneous BGK Model
\thanks{
  This manuscript has been authored, in part, by UT-Battelle, LLC, under contract DE-AC05-00OR22725 with the US Department of Energy (DOE). The U.S. government retains and the publisher, by accepting the article for publication, acknowledges that the U.S. government retains a nonexclusive, paid-up, irrevocable, worldwide license to publish or reproduce the published form of this manuscript, or allow others to do so, for U.S. government purposes. DOE will provide public access to these results of federally sponsored research in accordance with the DOE Public Access Plan (https://energy.gov/downloads/doe-public-access-plan).
  }
}

\author{
Evan Habbershaw\thanks{
    Department of Mathematics, The University of Tennessee, Knoxville, TN, 37996
  }, 
Cory D. Hauck\thanks{
    Department of Mathematics, The University of Tennessee, Knoxville, TN, 37996
    and
    Computer Science and Mathematics Division, Oak Ridge National Laboratory, Oak Ridge, TN, 37831
  }, and 
Steven M. Wise\thanks{
  Department of Mathematics, The University of Tennessee,  Knoxville, TN, 37996
  }
}

\begin{document}
    \pagenumbering{arabic}
    
    \maketitle

    \begin{abstract}
\noindent A simple iterative approach for solving a set of implicit kinetic moment equations is proposed.  This implicit solve is a key component in the IMEX discretization of the multi-species Bhatnagar-Gross-Krook (M-BGK) model with nontrivial collision frequencies depending on individual species temperatures. We prove that under mild time step restrictions, the iterative method generates a contraction mapping.  Numerical simulations are provided to illustrate results of the IMEX scheme using the implicit moment solver.
    \end{abstract}
    

    
    \section{Introduction}

Bhatnagar-Gross-Krook (BGK) models (see e.g. \cite{liboff2003kinetic}) are a computationally inexpensive alternative to the Boltzmann equation of rarefied gas dynamics.
They are derived by replacing the sophisticated Boltzmann collision operator with a nonlinear relaxation model.  
In the single species case, the target of this relaxation model is a Maxwellian distribution that shares the same mass, momentum, and energy density as the kinetic distribution.
This model was first introduced in \cite{bhatnagar1954}, and rigorous mathematical results about the model can be found in \cite{perthame_1989,perthame_1993}.

Like the Boltzmann equation, the single-species BGK model recovers the compressible system of Euler equations in the limit of infinite collisions \cite{cercignani2013mathematical}.
For simulation purposes, implicit-explicit (IMEX) methods are often employed that treat the advection term explicitly and the stiff collision term implicitly.
The strategy for the implicit evaluation of the BGK operator is to recognize that the moments that characterize the Maxwellian target are conserved by the space-homogeneous problem.
As a result, the implicit part of the relaxation operator is reduced to a simple multiple of the identity.
This trick, introduced in \cite{coron}, is the basis of many semi-implicit schemes.

More recently, multi-species versions of the BGK equation have been derived that satisfy the conservation and entropy dissipation properties of the multi-species Boltzmann equation.
Some models still use a single relaxation operator \cite{andries2002consistent,brull2012derivation} while others use a sum of operators with Maxwellian-like targets for each binary interaction \cite{bobylev2018general,klingenberg2017consistent,haack2017model,haack2021consistent}.
We focus here on the model from \cite{haack2017model}, but the approach and analysis can be extended to the generalizations in \cite{bobylev2018general,klingenberg2017consistent}.
In these models, each relaxation operator is equipped with a collision frequency.
Conservation of the number density in the multi-species setting implies that frequencies that depend only on the number densities will not be affected by the collision process.
However, for general molecular models, collision frequencies may depend on the species temperature, average velocities, and even microscopic velocities \cite{book:Bird_DSMC}, and removing such dependencies can significantly alter the behavior of solutions \cite{haack2021consistent}.
In addition, the analysis and simulation of the BGK equations becomes more difficult when the frequencies depend on time and/or phase-space velocity.

The focus of the present paper is the implicit update of the collision term in the multi-species setting when more realistic collision frequencies are employed.
In order to follow the strategy used for the single-species model, one must update an ODE system for the species' average velocities and temperatures, thereby enabling the evaluation of the Maxwellian target at the next time step.
As in the single species case, this approach effectively linearizes the problem, allowing for a simple update of the collision term.
However, unlike the single species case, the moments of each species are not conserved; rather the individual species densities, total momentum, and total energy are conserved.
As a result, the ODE for the average velocities and temperatures requires an implicit solve.
For collision frequencies that depend only on species densities, such a solve is still trivial to implement because the frequencies are not affected by the collisions.
However, for more complicated collision frequencies, a fully implicit solve must take such variations into account.

We propose here to solve the multi-species BGK collision operator in a fully implicit manner using a nonlinear Gauss-Seidel-type iterative solver.
The focus of the paper is on the numerical analysis of the solver.
In particular, we demonstrate a contraction mapping under mild time step restrictions.
Importantly, we show that, given a user specified time step, there is a stable time step that, at worst, is smaller by an $\mathcal{O}(1)$ constant that does not depend on the stiffness of the collision operator.
Thus, in the fluid regime, the method can still take time steps determined by particle advection and not particle collisions.
This result and the underlying analysis can be extended to general relaxation systems with state-dependent relaxation times.

The remainder of the paper is organized as follows.
In \Cref{section:background}, we introduce the model and discuss the IMEX strategy that provides the context for the implicit solver.
In \Cref{section:solver}, we introduce the solver.
In \Cref{section:velocityProof} and \Cref{section:temperatureProof}, we analyze the convergence of the solver, and in \Cref{section:timeStepSelection}, we analyze the associated time step restriction.
In \Cref{section:numerics}, we provide several numerical demonstrations, and in \Cref{section:conclusions} we make conclusions and discuss future work.
The Appendix contains several results used in the main body of the paper.

    \section{Background}
    \label{section:background}
    
    \subsection{Models and equations}
    
We denote by $f_i$ the kinetic distribution of particles of species $i\in\{1,\cdots,N\}$ having mass $m_i>0$.
More specifically, $f_i(\bfx,\bfv,t)$ is the density of species $i$ particles at the point $\bfx\in\Omega\subset\bbR^d$, with microscopic velocity $\bfv\in\bbR^d$, at time $t\geq0,$ with respect to the measure $\diff\bfv\,\diff\bfx$.
Associated to each $f_i$ are the species number density $n_i$, bulk velocity $\bfu_i$, and temperature $T_i$, defined by
    \begin{equation}
n_i = \intd f_i \dv
  ,\qquad
\bfu_i = \frac{1}{n_i}\intd\bfv f_i \dv
  ,\qquad
T_i  = \frac{m_i}{n_i d}\intd|\bfv-\bfu_i|^2f_i \dv
.
    \end{equation}
The species mass density $\rho_i$, momentum density $\bfk_i$, and energy density $E_i$ are
    \begin{equation}
\rho_i = m_in_i
  ,\qquad  
\bfk_i = \rho_i \bfu_i
  ,\qquand
E_i = \frac{1}{2}\rho_i|\bfu_i|^2 + \frac{d}{2}n_i T_i
.
    \end{equation}
    
BGK models are expressed in terms of Maxwellian distributions that depend on these moments.

    \begin{defn}
    \label{defn:Maxwellian}
Given $n>0$, $\bfu\in\bbR^d$, and $\theta>0$, the Maxwellian $M_{n,\bfu,\theta}$ is a function of the form
    \begin{equation}
    \label{eq:Maxwellian_def}
M_{n,\bfu,\theta}(\bfv)
  = \frac{n}{(2\pi\theta)^{d/2}}
  \exp
    \left(
      -\frac{|\bfv-\bfu|^2}{2\theta}
    \right)
.
    \end{equation}
    \end{defn}
    
Straightforward computations show that the moments of $M_{n,\bfu,\theta}$ satisfy
    \begin{equation}
\intd M_{n,\bfu,\theta}\,\diff\bfv = n,
    \qquad
\intd\bfv M_{n,\bfu,\theta} \, \diff\bfv = n\bfu,
    \qquad
\intd|\bfv|^2M_{n,\bfu,\theta} \, \diff\bfv = n|\bfu|^2+dn\theta
.
    \end{equation}

We consider in this paper the multi-species BGK equation from \cite{haack2017model,asymptoticRelaxation}, given by
    \begin{equation}
    \label{eq:BGKequation}
\frac{\partial f_i}{\partial t}
+
\bfv\cdot\nabla_{\bfx}f_i
    =
      \frac{1}{\veps}\sum_{j=1}^N\lambda_{i,j}(M_{i,j}-f_i),
\qquad\forall i\in\{1,\cdots,N\}
,
    \end{equation}
where $\veps>0$ is a dimensionless parameter that comes from a rescaling of space and time, $\lambda_{i,j}>0$ is the frequency of collisions between species $i$ and $j$ (independent of $\bfv$), and
    \begin{equation}
M_{i,j}(\bfv)
  \coloneqq
    M_{n_i,\bfu_{i,j},T_{i,j}/m_i}(\bfv) 
    \end{equation}
is a Maxwellian defined by \eqref{eq:Maxwellian_def}, using the mixture velocities and temperatures
    \begin{subequations}
    \label{eq:MomentDefs}
    \begin{align}
\bfu_{i,j}
  &=
    \frac
        {\rho_i\lambda_{i,j}\bfu_i + \rho_j\lambda_{j,i}\bfu_j}
        {\rho_i\lambda_{i,j} + \rho_j\lambda_{j,i}},
    \\
T_{i,j}
  &=
    \frac
        {n_i\lambda_{i,j}T_i + n_j\lambda_{j,i}T_j}
        {n_i\lambda_{i,j} + n_j\lambda_{j,i}}
    +
    \frac{1}{d}
    \frac
        {\rho_i\rho_j\lambda_{i,j}\lambda_{j,i}}
        {\rho_i\lambda_{i,j}+\rho_j\lambda_{j,i}}
    \frac
        {|\bfu_i-\bfu_j|^2}
        {n_i\lambda_{i,j}+n_j\lambda_{j,i}}
.
    \end{align}
    \end{subequations}
As in \cite{asymptoticRelaxation}, we consider a hard spheres collision model for the collision frequencies:
    \begin{equation}
    \label{eq:collisionFrequencies}
\lambda_{i,j}
  \coloneqq
    \lambda_{i,j}^{\textnormal{HS}}(n_j,T_i,T_j)
  =
    \frac{32 \pi^2}{3(2\pi)^{3/2}} \frac{m_i m_j}{(m_i + m_j)^2}(d_i + d_j)^2n_j\sqrt{\frac{T_i}{m_i} + \frac{T_j}{m_j}},
    \end{equation}
where $d_i$ and $d_j$ are reference diameters for the particles of species $i$ and $j$.
For simplicity, we write $\lambda_{i,j}= c_{i,j}n_j\Psi_{i,j},$ where 
    \begin{equation}
    \label{eq:Psi}
\Psi_{i,j}(T_i,T_j)
  =
    \sqrt{\frac{T_i}{m_i}+\frac{T_j}{m_j}}
    \end{equation}
is the only time-dependent term.
The challenging features of $\lambda_{i,j}^{\textnormal{HS}}$ are (i) it depends on species temperatures that are not conserved by collision dynamics, and (ii) this temperature dependence is not Lipschitz at zero temperatures.

Going forward, we will abbreviate summation symbols $\sum_{j=1}^N$ as $\sum_j$, when the upper and lower index is understood as $j=1$ and $j=N$, respectively.

    \subsection{IMEX strategy and backward Euler step}
    \label{section:imex}
    
Implicit-explicit (IMEX) Runge Kutta methods are a common tool for simulating kinetic equations.
These methods typically treat advection explicitly and collisions implicitly.
A diagonally implicit IMEX Runge-Kutta method for \eqref{eq:BGKequation} with $\nu$ stages and time step $\dt^{\mathfrak{n}}>0$ takes the form \cite{puppo2019kinetic}
    \begin{subequations}
    \label{eq:IMEX}
    \begin{align}
    \label{eq:IMEX_stages}
f_i^{(s)}
  &=
    f_i^{(s,\star)}
    +
    \frac{\dt^\mathfrak{n}}{\veps} a_{s,s} \sum_j
      \lambda_{i,j}^{(s)} (M^{(s)}_{i,j}-f^{(s)}_i)
,\qquad s \in \{1, \cdots, \nu\}
,
    \\
f^{\mathfrak{n}+1}_i
  &= 
    f_i^{\mathfrak{n}}
    -
    \dt^{\mathfrak{n}}
    \sum_{s=1}^\nu
      \widetilde{b}_s \bfv\cdot\nabla_{\bfx}f_i^{(s)}
    +
    \frac{\dt^{\mathfrak{n}}}{\veps}
    \sum_{s=1}^\nu
      b_s
      \sum_j
        \lambda_{i,j}^{(s)}(M_{i,j}^{(s)}-f_i^{(s)})
,
    \end{align}
    \end{subequations}
where
    \begin{equation}
f_i^{(s,\star)}
  = 
    f_i^{\mathfrak{n}}
    -
    \dt^{\mathfrak{n}}
    \sum_{r=1}^{s-1}
      \widetilde{a}_{s,r} \bfv \cdot \nabla_{\bfx}  f^{(r)}_i
    +
    \frac{\dt^{\mathfrak{n}}}{\veps}
    \sum_{r=1}^{s-1}
      a_{s,r}
      \sum_j
        \lambda_{i,j}^{(r)} (M^{(r)}_{i,j}-f^{(r)}_i)
,
    \end{equation}
and
$\lambda_{i,j}^{(r)} = \lambda_{i,j}^{\textnormal{HS}}(n_j^{(r)},T^{(r)}_i,T^{(r)}_j)$.
The coefficients $a_{s,r}$ and $\widetilde{a}_{s,r}$ are determined by order conditions and other desired properties of the method \cite{pareschi2005implicit,ascher1997implicit,boscarino2017unified}.

A prototype for the update in \eqref{eq:IMEX_stages}, and the focus of this paper, is the Backward Euler update, namely,
    \begin{equation}
    \label{eq:IMEX_BE}
f_i^{\mathfrak{n}+1}
  =
    f_i^{\mathfrak{n}}
    +
    \frac{\dt}{\veps}
    \sum_j
      \lambda_{i,j}^{\mathfrak{n}+1}(M^{\mathfrak{n}+1}_{i,j}-f^{\mathfrak{n}+1}_i)
.
    \end{equation}
Indeed, given $f_i^{(s,\star)}$, solving \eqref{eq:IMEX_BE} is equivalent to solving \eqref{eq:IMEX} with slightly modified constants.
The standard approach to solve the nonlinear equation \eqref{eq:IMEX_BE} is to first solve moment equations for $n^{\mathfrak{n}+1}_i$, $\bfu^{\mathfrak{n}+1}_i$, and $T^{\mathfrak{n}+1}_i$, which can then be used to define $M^{\mathfrak{n}+1}_{i,j}$ in an explicit fashion; what remains is a linear solve for $f_i^{\mathfrak{n}+1}$ \cite{puppo2019kinetic}.
In the single species case, the update is trivial: due to conservation properties of the collision operator $(n_i^{\mathfrak{n}+1},\bfu_i^{\mathfrak{n}+1},T_i^{\mathfrak{n}+1})=(n_i^{\mathfrak{n}},\bfu_i^{\mathfrak{n}},T_i^{\mathfrak{n}})$ \cite{coron}.
In the multi-species case, taking moments of \eqref{eq:IMEX_BE} and applying the definitions of $\bfu_{i,j}$ and $T_{i,j}$ from \eqref{eq:MomentDefs} gives the following algebraic system:

    \begin{prop}[Backward Euler]
If scheme \eqref{eq:IMEX_BE} is applied for the kinetic densities, then, for each $i\in\{1,\cdots,N\}$,
    \begin{subequations}
    \label{eq:BEstep}
    \begin{align}
    \label{eq:BE_number} 
\frac{n_{i}^{\mathfrak{n}+1}-n_i^{\mathfrak{n}}}{\dt}
  &= 0
,
    \\
    \label{eq:BE_velocity}
\rho_i\frac{\bfu_{i}^{\mathfrak{n}+1}-\bfu_i^{\mathfrak{n}}}{\dt}
  &=
    \frac{1}{\veps}\sum_j
    \left(
      A_{i,j}^{\mathfrak{n}+1}
      -
      D_{i,j}^{\mathfrak{n}+1}
    \right)
    \bfu_{j}^{\mathfrak{n}+1},
    \\
    \label{eq:BE_temperature}
    \begin{split}
\rho_i\frac{s_i^{\mathfrak{n}+1}-s_i^{\mathfrak{n}}}{\dt}
+
dn_i \frac{T_i^{\mathfrak{n}+1}-T_i^{\mathfrak{n}}}{\dt}
  &=
    \frac{d}{\veps}\sum_j
    \left(
      B_{i,j}^{\mathfrak{n}+1}
      -
      F_{i,j}^{\mathfrak{n}+1}
    \right)
    T_j^{\mathfrak{n}+1}
    \\
+\frac{1}{\veps}\sum_jB_{i,j}^{\mathfrak{n}+1}
\!&\,
    \left[
      m_j(s_j^{\mathfrak{n}+1}-S_{i,j}^{\mathfrak{n}+1})
      -
      m_i(s_i^{\mathfrak{n}+1}-S_{i,j}^{\mathfrak{n}+1})
    \right]
,
    \end{split}
    \end{align}
   \end{subequations}
where $s_i^\mathfrak{n} = |\bfu_i^\mathfrak{n}|^2$ and the ${N\times N}$ matrices $A$, $B$, $D$, $F$, and $S$, at time step $t^{\mathfrak{n}+1}$ are given by
    \begin{subequations}
    \begin{gather}
[A]_{i,j}^{\mathfrak{n}+1}
  = 
    \frac
        {\rho_i\rho_j\lambda_{i,j}^{\mathfrak{n}+1}\lambda_{j,i}^{\mathfrak{n}+1}}
        {\rho_i\lambda_{i,j}^{\mathfrak{n}+1}+\rho_j\lambda_{j,i}^{\mathfrak{n}+1}}
,\qquad
[B]_{i,j}^{\mathfrak{n}+1}
  =
    \frac
        {n_in_j\lambda_{i,j}^{\mathfrak{n}+1}\lambda_{j,i}^{\mathfrak{n}+1}}
        {n_i\lambda_{i,j}^{\mathfrak{n}+1}+n_j\lambda_{j,i}^{\mathfrak{n}+1}}
,\qquad
[S]_{i,j}^{\mathfrak{n}+1}
  =
    \left|
      \bfu_{i,j}^{\mathfrak{n}+1}
    \right|^2
,
    \\
[D]_{i,j}^{\mathfrak{n}+1}
  =
    \Big(
      \sum_k A_{i,k}^{\mathfrak{n}+1}
    \Big)
    \delta_{i,j}
,\qquad
[F]_{i,j}^{\mathfrak{n}+1}
  =
    \Big(
      \sum_k B_{i,k}^{\mathfrak{n}+1}
    \Big)
    \delta_{i,j}
.
    \end{gather}
    \end{subequations}
    \end{prop}
    
    \begin{remark}
According to \eqref{eq:BE_number}, the quantities $n_i$ and $\rho_i$ are independent of $\mathfrak{n}$.
We therefore do not include the superscript on these variables.
    \end{remark}

We summarize some useful properties of the backward Euler method below.
    
    \begin{prop}[Discrete Conservation Laws]
    \label{prop:discreteConsLaws}
The backward Euler scheme in \eqref{eq:BEstep} satisfies discrete versions of the conservation of total momentum and conservation of total energy.
Specifically,
    \begin{equation}
    \label{eq:discreteConsLaws}
\sum_i\rho_i\bfu_i^{\mathfrak{n}+1}
  =
    \sum_i\rho_i\bfu_i^{\mathfrak{n}} 
\qquand
\sum_iE_i^{\mathfrak{n}+1}
  =
    \sum_iE_i^{\mathfrak{n}}
.
    \end{equation}
    \end{prop}
    
    \begin{proof}
Sum \eqref{eq:BE_velocity} and \eqref{eq:BE_temperature} over all species $i$, and use the symmetry of $A_{i,j}^{\mathfrak{n}+1}$ and $B_{i,j}^{\mathfrak{n}+1}$ to obtain the result.
For example, summing \eqref{eq:BE_velocity} on $i$, using the definition of $D_{i,j}^{\mathfrak{n}+1}$, and symmetry of $A_{i,j}^{\mathfrak{n}+1}$ gives $\frac{1}{\veps}\sum_{i,j}(A_{i,j}^{\mathfrak{n}+1}-D_{i,j}^{\mathfrak{n}+1})\bfu_j^{\mathfrak{n}+1} = \frac{1}{\veps}\sum_{i,j}A_{i,j}^{\mathfrak{n}+1}(\bfu_j^{\mathfrak{n}+1}-\bfu_i^{\mathfrak{n}+1}) = 0$.
    \end{proof}
    
    \begin{prop}
The quantities
    \begin{equation}
\bfu^\infty
  = \frac{\sum_i\rho_i\bfu_i^{\mathfrak{n}}}{\sum_i\rho_i}
\qquand
T^\infty
  = \frac{\sum_in_iT_i^{\mathfrak{n}}}{\sum_in_i}+\frac{\sum_i\rho_i\left(|\bfu_i^{\mathfrak{n}}|^2-|\bfu^\infty|^2\right)}{d\sum_in_i}
    \end{equation}
are independent of the time step index $\mathfrak{n}$.
    \end{prop}

    \begin{proof}
The fact that $\bfu^\infty$ is independent of $\mathfrak{n}$ follows immediately from the first equation in \eqref{eq:discreteConsLaws}.
Meanwhile, 
    \begin{equation}
T^\infty
  =
    \frac{2\sum_iE_i^{\mathfrak{n}}-\sum_i\rho_i|\bfu^\infty|^2}{d\sum_in_i}
    \end{equation}
is independent of $\mathfrak{n}$ according to the second equation in \eqref{eq:discreteConsLaws} and the fact that $\bfu^\infty$ is also independent of $\mathfrak{n}$.
    \end{proof}

When  collision frequencies are constant or depend only on the number densities, $\lambda_{i,j}^{\mathfrak{n}+1} = \lambda_{i,j}^{\mathfrak{n}}$, and solving \eqref{eq:BEstep} is straightforward \cite{puppo2019kinetic}.
First, \eqref{eq:BE_velocity} requires a linear solve for $\bfu^{\mathfrak{n}+1}_i$.
The new velocity can then be used to evaluate implicit source terms in \eqref{eq:BE_temperature}, leaving another linear solve to find each $T_i$.
However, for more realistic frequencies that depend on $T_i$, this approach no longer works.
Two options remain: (i) lag $\lambda_{i,j}$ at the previous time step or (ii) solve \eqref{eq:BEstep} in a fully nonlinear fashion.  

Since the first option may lead to loss of accuracy or instability, we focus here on the second option.
In terms of cost, the linear solve is trivial compared to the rest of the kinetic simulation.
The challenge is to establish an iterative algorithm that is provably convergent, even when $\dt / \veps$ is large.

    \section{Iterative Solver}
    \label{section:solver}
    
    \subsection{Definition and Main Result}
    
We propose the following Gauss-Seidel type (GST) iterative solver to solve the backward Euler system in \eqref{eq:BEstep}.
We show below that, under a mild time step restriction which is independent of $\veps$, the GST iteration has a fixed point that yields the unique solution to \eqref{eq:BEstep}.
    
    \begin{defn}[GST method]
Let $\bfu_i^{\mathfrak{n}+1,0} = \bfu_i^{\mathfrak{n}}$ and $T_i^{\mathfrak{n}+1,0} = T_i^{\mathfrak{n}}$ be given.
For $\ell\in \{0,1,2,\cdots\}$, let $\bfu_i^{\mathfrak{n}+1,\ell+1}$ and $T_i^{\mathfrak{n}+1,\ell+1}$ be given by
    \begin{subequations}
    \label{eq:GSTstep}
    \begin{align}
    \label{eq:GSTstep_vel}
\rho_i\frac{\bfu_{i}^{\mathfrak{n}+1,\ell+1}-\bfu_i^{\mathfrak{n}}}{\dt}
  &= 
    \frac{1}{\veps}\sum_j
    \left(
      A_{i,j}^{\mathfrak{n}+1,\ell}
      -
      D_{i,j}^{\mathfrak{n}+1,\ell}
    \right)
    \bfu_{j}^{\mathfrak{n}+1,\ell+1}
,
    \\
    \label{eq:GSTstep_temp}
    \begin{split}
\rho_i\frac{s_i^{\mathfrak{n}+1,\ell+1}-s_i^{\mathfrak{n}}}{\dt}
+ dn_i \frac{T_i^{\mathfrak{n}+1,\ell+1}-T_i^{\mathfrak{n}}}{\dt}
  &= 
    \frac{d}{\veps}\sum_j
    \left(
      B_{i,j}^{\mathfrak{n}+1,\ell}
      -
      F_{i,j}^{\mathfrak{n}+1,\ell}
    \right)
    T_j^{\mathfrak{n}+1,\ell+1}
    \\
  & \hspace{-5cm}+ \frac{1}{\veps}\sum_jB_{i,j}^{\mathfrak{n}+1,\ell}
    \left[m_j
      \left(
        s_j^{\mathfrak{n}+1,\ell+1}-S_{i,j}^{\mathfrak{n}+1,\ell+1}
      \right)
      -m_i
      \left(
        s_i^{\mathfrak{n}+1,\ell+1}-S_{i,j}^{\mathfrak{n}+1,\ell+1}
      \right)
    \right]
.
    \end{split}
    \end{align}
    \end{subequations}
    \end{defn}

To simplify notation, we drop the superscript $\mathfrak{n}+1$, since it is clear that we are seeking the update at this time level.
Thus, instead of $\bfu_i^{\mathfrak{n}+1,\ell}$ we use $\bfu_i^{\ell}$, and, likewise, $T_i^{\mathfrak{n}+1,\ell}$ is replaced by $T_i^{\ell}$.
In particular, the initialization $\bfu_i^{\mathfrak{n}+1,0} \coloneqq \bfu_i^{\mathfrak{n}}$ is replaced by $\bfu_i^{\ell=0} \coloneqq \bfu_i^{\mathfrak{n}}$, and $T_i^{\mathfrak{n}+1,0} \coloneqq T_i^{\mathfrak{n}}$ is replaced by $T_i^{\ell=0} \coloneqq T_i^{\mathfrak{n}}$.
Furthermore, to facilitate the analysis of the GST method, we introduce the auxiliary variables
    \begin{equation}
\bfW
  =
    [\bfw_1,\cdots,\bfw_N]^ \top
  =
    P^{\sfrac{1}{2}}\bfU
\qquand
\bfeta
  =
    Q^{\sfrac{1}{2}}\bfT
,
    \end{equation}
where $P = \text{diag}\{\rho_k\} \in\bbR^{N\times N}$, $Q = \text{diag}\{n_k\}\in\bbR^{N\times N}$, $\bfU = [\bfu_1,\cdots,\bfu_N]^\top\in\bbR^{N\times d}$, and $\bfT = (T_1,\cdots,T_N)^\top\in\bbR^N$.

    \begin{prop}[GST Method for $\bfW$ and $\bfeta$]  
The iterates $\bfW^\ell = P^{\sfrac{1}{2}}\bfU^\ell$ and $\bfeta^\ell = Q^{\sfrac{1}{2}}\bfT^\ell$ satisfy
    \begin{subequations}
    \label{eq:GaussSeidelIterates}
    \begin{align}
    \label{eq:WUpdateIterative}
\left(
  I+\frac{\dt}{\veps}Z^\ell
\right)
\bfW^{\ell+1}
  &=
    \bfW^{\mathfrak{n}}
,
    \\
    \label{eq:ETAUpdateIterative}
\left(
  I+\frac{\dt}{\veps}\what{Z}^\ell
\right)
\bfeta^{\ell+1}
  &=
    \bfeta^{\mathfrak{n}}
    +
    \frac{1}{d}Q^{-\sfrac{1}{2}}\bfs^{(1),\ell+1}
    +
    \frac{\dt}{\veps}\frac{1}{d}Q^{-\sfrac{1}{2}}\bfs^{(2),\ell+1,\ell}
,
    \end{align}
    \end{subequations}
where the $N \times N$ matrices
    \begin{equation}
Z^\ell
  =
    P^{-\sfrac{1}{2}} (D^\ell-A^\ell) P^{-\sfrac{1}{2}}
\qquand
\what{Z}^\ell
  =
    Q^{-\sfrac{1}{2}} (F^\ell-B^\ell) Q^{-\sfrac{1}{2}}
    \end{equation}
are symmetric, and the source terms in \eqref{eq:ETAUpdateIterative} are given by
    \begin{subequations}
    \begin{align}
    \label{eq:source_term1}
\bfs^{(1),\ell+1}
  &=
    \bfs_{\bfW}^{\mathfrak{n}} - \bfs_{\bfW}^{\ell+1}
,\qquad \text{with }\qquad
\bfs_{\bfW}
  = 
    \left(
      |\bfw_1|^2,\cdots,|\bfw_N|^2
    \right)^\top
,
    \\
    \label{eq:source_term2}
\text{and }\qquad
\left[
  \bfs^{(2),\ell+1,\ell}
\right]_i
  &=
    \sum_jB_{i,j}^\ell
    \left[
      m_j(s_j^{\ell+1}-S_{i,j}^{\ell+1})
      -
      m_i(s_i^{\ell+1}-S_{i,j}^{\ell+1})
    \right]
.
    \end{align}
    \end{subequations}
    \end{prop}

    \begin{defn}
    \label{definition:A_ij_B_ij}
Let
    \begin{equation}
c_{i,j}^A
  =
    \frac{16}{3}\sqrt{\frac{\pi}{2}}
    \frac{m_im_j(d_i+d_j)^2}{(m_i+m_j)^3}\rho_i\rho_j
\qquand
c_{i,j}^B
  =
    \frac{8}{3}\sqrt{\frac{\pi}{2}}
    \frac{(d_i+d_j)^2}{(m_i+m_j)^2}\rho_i\rho_j
,
    \end{equation}
so that $A^{\mathfrak{n}}_{i,j} = c_{i,j}^A\Psi_{i,j}^{\mathfrak{n}}$ and $B^{\mathfrak{n}}_{i,j} = c_{i,j}^B\Psi_{i,j}^{\mathfrak{n}}$, with $\Psi$ as in \eqref{eq:Psi}.
Also define the bounds
    \begin{subequations}
    \begin{align}
A_{\min}^{\mathfrak{n}}
    &\coloneqq
\min_{i,j}A_{i,j}^{\mathfrak{n}}
    \leq
A_{i,j}^{\mathfrak{n}}
    \leq
\max_{i,j} A_{i,j}^{\mathfrak{n}}
    \eqqcolon
A_{\max}^{\mathfrak{n}}
,
    \\
B_{\min}^{\mathfrak{n}}
    &\coloneqq
\min_{i,j}B_{i,j}^{\mathfrak{n}}
    \leq
B_{i,j}^{\mathfrak{n}}
    \leq
\max_{i,j}B_{i,j}^{\mathfrak{n}}
    \eqqcolon
B_{\max}^{\mathfrak{n}}
.
    \end{align}
    \end{subequations}
Finally, define the following maximum and minimum values:
    \begin{subequations}
    \begin{gather}
m_{\min} \coloneqq \min\{m_k\}
,\qquad
n_{\min} \coloneqq \min\{n_k\}
,\qquad
\rho_{\min} \coloneqq \min\{\rho_k\}
,
    \\
m_{\max} \coloneqq \max\{m_k\}
,\qquad
n_{\max} \coloneqq \max\{n_k\}
,\qquad
\rho_{\max} \coloneqq \max\{\rho_k\}
.
    \end{gather}
    \end{subequations}
    \end{defn}
    
    \begin{prop}
    \label{prop:z_min_defs}
The eigenvalues of the matrices $Z^\ell=P^{-\sfrac{1}{2}}(D^\ell-A^\ell)P^{-\sfrac{1}{2}}$ and $\what{Z}^\ell=Q^{-\sfrac{1}{2}}(F^\ell-B^\ell)Q^{-\sfrac{1}{2}}$ satisfy
    \begin{subequations}
    \begin{align}
0 
  &=
    z_0
    < z_{\min}^{\mathfrak{n}}
    \leq z_1^\ell
    \leq \cdots
    \leq z_{N-1}^\ell
    \leq z_{\max}^\mathfrak{n}
,
    \\
0 
  &=
    \what{z}_0
    < \what{z}_{\min}^\mathfrak{n}
    \leq \what{z}_1^\ell
    \leq \cdots 
    \leq \what{z}_{N-1}^\ell
    \leq \what{z}_{\max}^\mathfrak{n}
,
    \end{align}    
    \end{subequations}
where
    \begin{equation}
z_{\min}^{\mathfrak{n}}
  = \frac{A_{\min}^{\mathfrak{n}}N}{\rho_{\max}}
    ,\qquad
z_{\max}^{\mathfrak{n}}
  = \frac{A_{\max}^{\mathfrak{n}}(N-1)}{\rho_{\min}} 
    ,\qquad
\what{z}_{\min}^{\mathfrak{n}}
  = \frac{B_{\min}^{\mathfrak{n}}N}{n_{\max}}
    ,\qquad
\what{z}_{\max}^{\mathfrak{n}}
  = \frac{B_{\max}^{\mathfrak{n}}(N-1)}{{n_{\min}}}
.
    \end{equation}
    \end{prop}

    \begin{proof}
The proof of this result follows as in \cite[Appendix B]{asymptoticRelaxation}, and uses \eqref{eq:tempBound}.
    \end{proof}
    
    \begin{defn}
    \label{defn:CauchyDelta}
For any quantity $K$ that depends on $\ell$, define the difference operator $\Delp$ by $\Delp{K} \equiv \Delp{(K)} = K^{\ell+1} - K^\ell$.
    \end{defn}
The main result of this paper is the following theorem.

    \begin{theorem}
    \label{theorem:mainTheorem_Cauchy}
The differences $\Delp{\bfW}$ and $\Delp{\bfeta}$, defined according to \eqref{eq:GaussSeidelIterates}, satisfy
    \begin{subequations}
    \begin{align}
    \label{eq:main_theorem_W}
\froNorm{\Delp{\bfW}}
  &\leq
    C_{\bfW}^{\mathfrak{n}}\Gamma_{\bfW}^\mathfrak{n}\|\Delm{\bfeta}\|_2
,
    \\
    \label{eq:main_theorem_eta}
\|\Delp{\bfeta}\|_2
  &\leq
    C_0^{\mathfrak{n}}\Gamma_{\bfW}^\mathfrak{n}\|\Delm{\bfeta}\|_2 
    +
    C_1^{\mathfrak{n}}\Gamma_{\bfX}\|\Delm{\bfeta}\|_2
,
    \end{align}     
\end{subequations}
where the positive constants $C_{\bfW}^\mathfrak{n}$, $C_0^\mathfrak{n}$, and $C_1^\mathfrak{n}$ depend only on $\left\{N,m_k,d_k,n_k,\bfW^{\mathfrak{n}},\bfeta^{\mathfrak{n}}\right\}$;
    \begin{equation}
\Gamma_{\bfW}^\mathfrak{n}
  =
    \frac{\dt\veps}{(\veps+\dt z_{\min}^\mathfrak{n})^2}
,\qquad
\Gamma_{\bfeta}^\mathfrak{n}
  =
    \frac{\dt\veps}{(\veps+\dt\what{z}_{\min}^\mathfrak{n})^2}
,\qquand
\Gamma_{\bfX}^\mathfrak{n}
  =
    \max\{\Gamma_{\bfW}^\mathfrak{n},\Gamma_{\bfeta}^\mathfrak{n}\}
;
    \end{equation}
and $z_{\min}^\mathfrak{n}$ and $\what{z}_{\min}^\mathfrak{n}$ are defined in \Cref{prop:z_min_defs}.
    \end{theorem}
As a corollary, we obtain the following sufficient condition on the time step to ensure the method is convergent.
Importantly, this condition does not depend on $\veps$.

    \begin{corollary}
    \label{theorem:mainTheorem_TimeStep}
Given $\dt_0>0$, there exists $\dt_1$ such that (i) $C\dt_0\leq\dt_1\leq\dt_0$ for a constant $C \in (0,1]$ that is independent of $\veps$, and (ii) for this $\dt_1$,
    \begin{equation}
C_{\bfW}^\mathfrak{n}\Gamma_{\bfW}^\mathfrak{n}
  + C_0^\mathfrak{n}\Gamma_{\bfW}^\mathfrak{n}
  + C_1^\mathfrak{n}\Gamma_{\bfX}^\mathfrak{n}
  <1
,
    \end{equation}
thus ensuring convergence of the GST method.
    \end{corollary}
Most of the remainder of the paper is dedicated to the proof of \Cref{theorem:mainTheorem_Cauchy}.
In the next subsection, we collect some useful properties of the method.
\Cref{section:velocityProof} is dedicated to proving the bound in \eqref{eq:main_theorem_W}, and \Cref{section:temperatureProof} is dedicated to proving the bound in \eqref{eq:main_theorem_eta}; see \Cref{theorem:VelocityError,theorem:TemperatureError}, respectively.
The proof of \Cref{theorem:mainTheorem_TimeStep} is contained in \Cref{section:timeStepSelection}.

    \subsection{Basic properties of the GST method}

    \begin{prop}
    \label{prop:GST_conservation_laws}
The iterates of the GST method satisfy the discrete version of the conservation of total momentum and total energy:
    \begin{equation}
\sum_i\rho_i\bfu_i^{\ell+1}
  =
    \sum_i\rho_i\bfu_i^{\mathfrak{n}}
\qquand
\sum_iE_i^{\ell+1}
  = 
    \sum_iE_i^{\mathfrak{n}}
.
    \end{equation}
    \end{prop}

    \begin{proof}
The proof follows as in \Cref{prop:discreteConsLaws}, with  $\mathfrak{n}+1$ replaced by $\ell$.
    \end{proof}
    
    \begin{prop}
    \label{prop:GST_properties}
The velocity and temperature iterates satisfy the bounds
    \begin{align}
    \label{eq:velocityBounds}
&\min_j\{u_{j,k}^{\mathfrak{n}}\}
  \leq
    \min_j\{u_{j,k}^{\ell+1}\}
  \leq
    \max_j\{u_{j,k}^{\ell+1}\}
  \leq
    \max_j\{u_{j,k}^{\mathfrak{n}}\}
,\qquad k\in\{1,\cdots,d\}
,
    \\
    \label{eq:tempBound}
&T_i^{\ell+1}
  \geq
    \min_j\{T_j^{\mathfrak{n}}\}
  \eqqcolon T_{\min}^{\mathfrak{n}}
,\qquad i\in\{1,\cdots,N\}
,
    \\
    \label{eq:u_max_def}
&\|\bfu_i^{\ell+1}\|_2
  \leq
    \|\bfu_{\max}^{\mathfrak{n}}\|_2
  \eqqcolon 
    u_{\max}^{\mathfrak{n}}
,\qquad i\in\{1,\cdots,N\}
,
    \end{align}
where $\left[\bfu_{\max}^{\mathfrak{n}}\right]_k = \max_j\{|u_{j,k}^{\mathfrak{n}}|\}$, for $k\in\{1,\cdots,d\}$.
    \end{prop}

    \begin{proof}
We first verify the lower bound in \eqref{eq:velocityBounds}.
The velocity update \eqref{eq:GSTstep_vel} is equivalent to
    \begin{equation}
u_{i,k}^{\ell+1}
  = u_{i,k}^{\mathfrak{n}}+ \frac{\dt}{\veps}\sum_j\lambda_{i,j}^{\ell}\alpha_{j,i}
    \left(
      u_{j,k}^{\ell+1}-u_{i,k}^{\ell+1}
    \right)
,
    \end{equation}
where $\alpha_{j,i} = \frac{\rho_j\lambda_{j,i}^\ell}{\rho_i\lambda_{i,j}^\ell+\rho_j\lambda_{j,i}^\ell}$ is independent of $\ell$ (see \Cref{lemma:alpha_beta_diffs}).
Write $\underline{u}_k^{\ell} = \min_j\{u_{j,k}^{\ell}\}$.
Then, with $c_i^{\ell} = \sum_j\lambda_{i,j}^{\ell}\alpha_{j,i}$,
    \begin{subequations}
    \begin{align}
    \label{eq:vel_bound_1}
u_{i,k}^{\ell+1}
  &\geq
    \underline{u}_k^{\mathfrak{n}} + \frac{\dt}{\veps}\sum_j\lambda_{i,j}^{\ell}\alpha_{j,i}
    \left(
      \underline{u}_k^{\ell+1}\!-u_{i,k}^{\ell+1}
    \right)
  =
    \underline{u}_k^{\mathfrak{n}}
    +
    \frac{\dt}{\veps}c_i^{\ell}\underline{u}_k^{\ell+1}
    -
    \frac{\dt}{\veps}c_i^{\ell}u_{i,k}^{\ell+1}
    \\
    \label{eq:vel_bound_1a}
\Longrightarrow\quad
u_{i,k}^{\ell+1}
  &\geq
    \frac{\veps}{\veps+\dt c_i^{\ell}}\underline{u}_k^{\mathfrak{n}}
    +
    \frac{\dt c_i^{\ell}}{\veps+\dt c_i^{\ell}}\underline{u}_k^{\ell+1}
.
    \end{align}
    \end{subequations}
Choose $i_{\star}\in\{1,\cdots,N\}$ such that $u_{i_{\star},k}^{\ell+1}=\underline{u}_k^{\ell+1}$, and let $c_\star^{\ell} = c_{i_\star}^{\ell}$.
Then after some simple algebra, \eqref{eq:vel_bound_1a} implies that
    \begin{equation}
\underline{u}_k^{\ell+1}
  \geq
    \frac{\veps}{\veps+\dt c_\star^{\ell}}\underline{u}_k^{\mathfrak{n}}
    +
    \frac{\dt c_\star^{\ell}}{\veps+\dt c_\star^{\ell}}\underline{u}_k^{\ell+1}
\qquad\Longrightarrow\qquad
\underline{u}_k^{\ell+1}
  \geq
    \underline{u}_k^{\mathfrak{n}}
.
    \end{equation}
The proof of the upper bound in \eqref{eq:velocityBounds} is similar, and
\eqref{eq:u_max_def} follows easily from \eqref{eq:velocityBounds}.

We next show \eqref{eq:tempBound}.
Rearranging \eqref{eq:GSTstep_temp} for the temperature gives
    \begin{align}
    \nonumber
T_i^{\ell+1}
  &= T_i^{\mathfrak{n}}
    + \frac{\dt}{\veps}\sum_j\lambda_{i,j}^{\ell}\beta_{j,i}
      \left(
        T_j^{\ell+1}-T_i^{\ell+1}
      \right)
    + \frac{m_i}{d}(s_i^{\mathfrak{n}}-s_i^{\ell+1})
    + \frac{\dt}{\veps d}
      \frac{1}{n_i}
      \left[
        \bfs^{(2),\ell+1,\ell}
      \right]_i
    \\
    \label{eq:temp_bound_1}
    &\geq T_i^{\mathfrak{n}}
    + \frac{\dt}{\veps}\sum_j\lambda_{i,j}^{\ell}\beta_{j,i}
    \left(
      T_j^{\ell+1}-T_i^{\ell+1}
    \right)
,
    \end{align}
where the inequality in \eqref{eq:temp_bound_1} follows from \Cref{lemma:TempSourcesPositive}.
The rest of the proof of \eqref{eq:tempBound} proceeds in the same way as the velocity lower bound, starting at \eqref{eq:vel_bound_1}. 
    \end{proof}

    \begin{remark}
The results of \Cref{prop:GST_properties} can be verified for the backward Euler step directly.
The proof is a slight modification of the one given above.
    \end{remark}

    \section{Analysis of Velocity Iterates}
    \label{section:velocityProof}

Let $\bfW_\mathcal{N}^{\ell}$ and $\bfW_\mathcal{R}^{\ell}$ be the null and range space components, respectively, of $\bfW^{\ell}$ with respect to $Z^\ell$.
Since $Z^\ell$ is symmetric, its null space and range space are orthogonal, and $\bfW^\ell=\bfW_\mathcal{N}^\ell + \bfW_\mathcal{R}^\ell$.
The following property will be used throughout the section.

    \begin{lemma}[\cite{asymptoticRelaxation}, Section 4]
    \label{lem:mullSpaceInvariance}
The null space $\mathcal{N}$ and range space $\mathcal{R}$ of the matrix $Z^\ell$ are independent of $\ell$.
The null space has dimension one and is spanned by the vector
$P^{\sfrac{1}{2}}\mathbf{1}$.
    \end{lemma}
    
It will be convenient for the analysis of velocity terms to use the Frobenius norm
    \begin{equation}
\froNorm{\bfW}^2
  = \sum_{i=1}^N\sum_{j=1}^dw_{i,j}^2
  = \sum_{i=1}^N\|\bfw_i\|_2^2
  = \sum_{j=1}^d
    \left\|
      \bfW_{(\cdot),j}
    \right\|_2^2
.
    \end{equation}

    \subsection{Velocity Iterates on Null Space}

    \begin{lemma}
    \label{lemma:nullSpaceIter}
The null space iterates are $\bfW^\ell_\mathcal{N} = P^{\sfrac{1}{2}}\bfU^\infty$, where $\bfU^\infty = \mathbf{1} (\bfu^\infty )^\top$.
In particular, they are independent of $\ell$.
As a result, $\froNorm{\Delp{\bfW}} = \froNorm{\Delp{(\bfW_\mathcal{R})}}$.
    \end{lemma}

    \begin{proof}
According to \Cref{lem:mullSpaceInvariance}, $\bfW_\mathcal{N}^\ell = P^{\sfrac{1}{2}}\mathbf{1}(\bfc^\ell)^\top$, for some $\bfc^\ell\in\bbR^{d}$.
Since $\bfW^\ell_\mathcal{N} \perp \bfW^\ell_\mathcal{R}$, it follows that $\mathbf{1}^\top P^{\sfrac{1}{2}} \bfW^\ell = \mathbf{1}^\top P^{\sfrac{1}{2}} \bfW^\ell_\mathcal{N} = \mathbf{1}^\top P \mathbf{1}(\bfc^\ell)^\top.$
Hence,
    \begin{equation}
\bfc^\ell
  = \frac{(\bfW^\ell)^\top P^{\sfrac{1}{2}} \mathbf{1}}{\mathbf{1}^\top P\mathbf{1}}
  = \frac{(\bfU^\ell)^\top P\mathbf{1}}{\mathbf{1}^\top P\mathbf{1}}
  = \frac{\sum_i\rho_i\bfu_i^\ell}{\sum_i\rho_i}
  = \bfu^\infty
,
    \end{equation}
which is independent of $\ell$ by \Cref{prop:GST_conservation_laws}.
Thus, $\bfW^{\ell}_\mathcal{N} = P^{\sfrac{1}{2}}\mathbf{1}(\bfu^\infty)^\top = P^{\sfrac{1}{2}} \bfU^{\infty}$ is independent of $\ell$.
As a consequence,
    \begin{equation}
\froNorm{\Delp{\bfW}}^2
  =
    \froNorm{\Delp{(\bfW_\mathcal{R}})}^2
    +
    \froNorm{\Delp{(\bfW_\mathcal{N}})}^2
  = 
    \froNorm{\Delp{(\bfW_\mathcal{R}})}^2
.
    \end{equation}
    \end{proof}

    \subsection{Velocity Iterates on Range Space and Initial Norm Bounds}
    
    \begin{defn}
    \label{definition:aIteratesVMatrix}
Let the columns of $V=[\bfv_1,\cdots,\bfv_{N-1}]\in\bbR^{N\times(N-1)}$ form an orthonormal basis of $\mathcal{R}$.
For each $\ell$, let $\bfA^\ell = [\bfa^\ell_1,\cdots,\bfa^\ell_{d}]\in\bbR^{(N-1)\times d}$ be the unique matrix of coefficients such that $\bfW^\ell_\mathcal{R}=V\bfA^\ell$.
     \end{defn}

    \begin{lemma}
    \label{lemma:simple_norms}
Given $V$ as in \Cref{definition:aIteratesVMatrix} and any $\bfY\in\bbR^{N\times d}$,
    \begin{equation}
    \label{eq:Vbounds}
\|V\|_2
  =
    1
,\qquad
\|V^\top\|_2
  \leq
    1
,\qquad
\froNorm{V^\top\bfY}
  \leq
    \froNorm{\bfY}
,
    \end{equation}
where $\|\cdot\|_2$ denotes the induced (operator) 2-norm when the argument is an operator.
In particular, setting $\bfY = \bfA$ and $\bfY =\Delp{\bfA}$, respectively, gives
    \begin{subequations}
    \label{eq:W}
    \begin{gather}
    \label{eq:W_bound}
\froNorm{\bfW^\ell_\mathcal{R}}
  =
    \froNorm{V\bfA^\ell}
  =
    \froNorm{\bfA^\ell}
,
    \\
    \label{eq:W_diff}
\froNorm{\Delp{\bfW}}
  =
    \froNorm{\Delp{(\bfW_\mathcal{R}})}
  =
    \froNorm{V\Delp(\bfA)}
  =
    \froNorm{\Delp{\bfA}}
.
    \end{gather}
    \end{subequations}
    \end{lemma}
Based on \eqref{eq:W}, we seek to bound $\froNorm{\bfA^\ell}$, and the difference $\froNorm{\Delp{\bfA}}$.

    \begin{lemma}
    \label{lemma:W_RangeSpaceIterates}
The matrix iterates $\bfA^{\ell+1}$ satisfy
    \begin{equation}
    \label{eq:W_RangeSpaceIterates}
\bfA^{\ell+1}
  =
    M^{\ell} V^\top\bfW^{\mathfrak{n}} 
\qquand
\Delp{\bfA}
  =
    \Delm{(M)} V^\top\bfW^{\mathfrak{n}}
,
    \end{equation}
where $M^{\ell} = \left(I+\frac{\dt}{\veps}V^\top Z^\ell V\right)^{-1}$.
    \end{lemma}
    
    \begin{proof}
Write $\bfW^{\ell+1}=\bfW_\mathcal{R}^{\ell+1}+\bfW_\mathcal{N}^{\ell+1} = V\bfA^{\ell+1}+\bfW_\mathcal{N}^{\ell+1}$ in \eqref{eq:WUpdateIterative} 
and multiply on the left by $V^\top$: 
    \begin{equation}
V^\top\left(I+\frac{\dt}{\veps}Z^\ell\right)
\left(V\bfA^{\ell+1}+\bfW_\mathcal{N}^{\ell+1}\right)
  = V^\top\bfW^{\mathfrak{n}}
.
    \end{equation}
Since $Z^\ell \bfW_\mathcal{N}^{\ell+1}=\mathbf{0}\in\bbR^{N\times d}$ and $V^\top\bfW_\mathcal{N}^{\ell+1}=\mathbf{0}\in\bbR^{(N-1)\times d}$, the result follows.
    \end{proof}

    \begin{lemma}
    \label{lemma:inverse_norms}
Let $\gamma^\mathfrak{n} = \frac{\veps}{\veps + \dt z_{\min}^\mathfrak{n}}$, where $z_{\min}^\mathfrak{n}>0$ is a lower bound on the positive eigenvalues of $Z^\mathfrak{n}$, defined in \Cref{prop:z_min_defs}.
Then
    \begin{equation}
    \label{eq:inverse_norms}
\|M^\ell\|_2
  \leq
    \gamma^\mathfrak{n}
\qquand
\|\Delm{M}\|_2
  \leq
    \Gamma_{\bfW}^\mathfrak{n}\|\Delm{{Z}}\|_2
,
    \end{equation}
where $\Gamma_{\bfW}^\mathfrak{n}=\frac{\dt}{\veps}(\gamma^\mathfrak{n})^2$.
    \end{lemma}

    \begin{proof}
The first bound in \eqref{eq:inverse_norms} is verified in \Cref{appendix:lemma:InverseMatrixNorm}.
For the second, the identity $X^{-1}-Y^{-1}=X^{-1}(Y-X)Y^{-1}$ with $X^{-1} = M^\ell$ and $Y^{-1} = M^{\ell-1}$ implies
    \begin{equation}
\Delm{M} 
  =
    - \frac{\dt}{\veps} M^\ell V^\top \Delm{(Z)} V M^{\ell-1}
.
    \end{equation}
The result follows after taking operator norms.
    \end{proof}
    
    \begin{lemma}
The matrix iterates $\bfA^\ell$ satisfy
    \begin{subequations}
    \begin{align}
    \label{eq:A_iterate_bound}
\froNorm{\bfA^{\ell+1}}
  &\leq
    \gamma^n \froNorm{\bfW^n} 
,
    \\
    \label{eq:A_difference_bound}
\froNorm{\Delta^\ell \bfA}
  &\leq
    \Gamma_{\bfW}^\mathfrak{n}
    \froNorm{\bfW^{\mathfrak{n}}}
    \|\Delm Z\|_2
,
    \end{align}
    \end{subequations}
where $\Gamma_{\bfW}^\mathfrak{n} = \frac{\dt}{\veps}(\gamma^\mathfrak{n})^2$.
    \end{lemma}
    
    \begin{proof}
The result follows using \Cref{lemma:simple_norms,lemma:W_RangeSpaceIterates,lemma:inverse_norms}.
    \end{proof}

    \begin{lemma}
The difference of matrices $Z^\ell$ satisfy
    \begin{equation}
    \label{eq:Z_diff_bound}
\|\Delm{Z}\|_2
  \leq
    C_Z^\mathfrak{n}\|\Delm{\bfeta}\|_2
,\qqwhere
C_Z^\mathfrak{n}
  =
    \frac
        {2(N+N^{\sfrac{1}{2}})C_A^\mathfrak{n}}
        {\rho_{\min}n_{\min}^{\sfrac{1}{2}}}
.
    \end{equation}
    \end{lemma}
    
    \begin{proof}
The proof is contained in \Cref{appendix:Z_and_ZHat_diffs}.
    \end{proof}
    
    \subsection{Summary of Velocity Term Analysis}
    
    \begin{theorem}
    \label{theorem:VelocityError}
The iterations of the velocity terms satisfy the following bound:
    \begin{equation}
\froNorm{\Delp{\bfW}}
  =
    \froNorm{\Delp{\bfA}}
  \leq
    C_{\bfW}^\mathfrak{n}\Gamma_{\bfW}^\mathfrak{n}
    \|\Delm{\bfeta}\|_2
,
    \end{equation}
where $C_{\bfW}^\mathfrak{n} = C_Z^\mathfrak{n}\froNorm{\bfW^{\mathfrak{n}}}$ and $\Gamma_{\bfW}^\mathfrak{n} = \frac{\dt\veps}{(\veps+\dt z_{\min}^\mathfrak{n})^2}$, with $C_Z^\mathfrak{n}$ as in \eqref{eq:Z_diff_bound} and $z_{\min}^\mathfrak{n}>0$ defined in \Cref{prop:z_min_defs}.
    \end{theorem}

    \begin{proof}
The result follows using \eqref{eq:W_diff}, \eqref{eq:A_difference_bound}, and \eqref{eq:Z_diff_bound}.
    \end{proof}

    \section{Analysis of  Temperature Iterations}
    \label{section:temperatureProof}
    
In this section, we analyze the temperature iterations; the results are summarized by \Cref{theorem:TemperatureError}.
The basic strategy is similar to that of the previous section, but there are important differences.
Denote the null space and range space of $\what{Z}^\ell$ by $\what{\mathcal{N}}$ and $\what{\mathcal{R}}$, respectively, and let $\bfeta^\ell=\bfeta_\mathcal{N}^\ell+\bfeta_\mathcal{R}^\ell$, where $\bfeta_\mathcal{N}^\ell \in \what{\mathcal{N}}$ and $\bfeta_\mathcal{R}^\ell \in \what{\mathcal{R}}$.
As before, the fact that $\what{\mathcal{R}} \perp \what{\mathcal{N}}$ implies that $\|\Delp{\bfeta}\|_2^2 = \|\Delp{\bfeta_\mathcal{N}}\|_2^2 + \|\Delp{\bfeta_\mathcal{R}}\|_2^2$. 
However, unlike $\bfW^\ell_\mathcal{N}$, the iterates $\bfeta^\ell_\mathcal{N}$ depend on $\ell$, and thus $\|\bfeta^{\ell+1}_\mathcal{N}-\bfeta^\ell_\mathcal{N}\|_2$ must be bounded.
In addition, there are source terms in the update for $\bfeta^\ell$ that depend on velocity iterates and require technical bounds.

    \subsection{Temperature Iterates on Null Space}

    \begin{lemma}[\cite{asymptoticRelaxation}, Sec. 4]
The null space $\what{\mathcal{N}}$ and range space $\what{\mathcal{R}}$ of $\what{Z}^\ell$ are independent of $\ell$.
The null space has dimension one and is spanned by the vector $Q^{\sfrac{1}{2}}\mathbf{1}$.
    \end{lemma}

    \begin{lemma}
    \label{lemma:TemperatureNullError}
The null space iterations $\bfeta^\ell_\mathcal{N}$ satisfy
    \begin{gather}
\left\|
  \Delp{(\bfeta_\mathcal{N})}
\right\|_2
  \leq
    C_0^\mathfrak{n}\Gamma_{\bfW}^\mathfrak{n}
    \left\|
      \Delm{\bfeta}
    \right\|_2,
\qqwhere
    \\
C_0^\mathfrak{n}
  =
    \frac{\sqrt{2}b}{d}\rho_{\max}^{\sfrac{1}{2}}u_{\max}^\mathfrak{n}
    N^{\sfrac{1}{2}}C_{\bfW}^\mathfrak{n}
,\qquad
b 
  =
    \|Q^{\sfrac{1}{2}}\mathbf{1}\|_2^{-1}
,\qquad
\Gamma_{\bfW}^\mathfrak{n}
  =
    \frac{\dt\veps}{(\veps+\dt z_{\min}^\mathfrak{n})^2}
.
    \end{gather}
    \end{lemma}

    \begin{proof}
Write $\bfeta_\mathcal{N}^{\ell+1}=b_0^{\ell+1}\what{\bfv}_0,$ where $ \what{\bfv}_0 = b Q^{\sfrac{1}{2}}\mathbf{1}$ satisfies $\|\what{\bfv}_0\|_2=1$.
Then
    \begin{equation}
\left\|
  \Delp{(\bfeta_\mathcal{N})}
\right\|_2
  =
    \left\|
      \Delp{(b_0)}\what{\bfv}_0
    \right\|_2
  =
    \left|
      \Delp{(b_0)}
    \right|
.
    \end{equation}
Now multiply \eqref{eq:ETAUpdateIterative} by $\what{\bfv}_0^\top$.
Since $\what{\mathcal{R}} \perp \what{\mathcal{N}}$, the left-hand side gives
    \begin{equation}
    \label{eq:b0_lhs}
\what{\bfv}_0^\top
  \left(
    I+\frac{\dt}{\veps}\what{Z}^\ell
  \right)
  (b_0^{\ell+1}\what{\bfv}_0+\bfeta_\mathcal{R}^{\ell+1}) 
  =
    b_0^{\ell+1} \what{\bfv}_0^\top \what{\bfv}_0 
  =
    b_0^{\ell+1}
.
    \end{equation}
Meanwhile, the right-hand side gives
    \begin{multline}
    \label{eq:b0_rhs}
\left(
  \what{\bfv}_0,\bfeta^{\mathfrak{n}}
\right)_2
  +
    \frac{1}{d}
    \left(
      \what{\bfv}_0,Q^{-\sfrac{1}{2}}
      \left(
        \bfs_{\bfW}^{\mathfrak{n}}-\bfs_{\bfW}^{\ell+1}
      \right)
    \right)_2
  +
    \frac{\dt}{\veps}\frac{1}{d}
    \left(
      \what{\bfv}_0,Q^{-\sfrac{1}{2}}\bfs^{(2),\ell+1,\ell}
    \right)_2
    \\
  =
    \left(
      \what{\bfv}_0,\bfeta^{\mathfrak{n}}
    \right)_2
  + 
    \frac{1}{d}
    \left(
      \what{\bfv}_0,Q^{-\sfrac{1}{2}}
      \left(
        \bfs_{\bfW}^{\mathfrak{n}}-\bfs_{\bfW}^{\ell+1}
      \right)
    \right)_2
,
    \end{multline}
where the second source term vanishes due to the symmetry of $B$.
Specifically, using \eqref{eq:source_term2} and the definition of $\what{\bfv}_0$,
    \begin{equation}
\left(
  \what{\bfv}_0,Q^{-\sfrac{1}{2}}\bfs^{(2),\ell+1,\ell}
\right)_2
  =
    b\sum_{i,j} B_{i,j}^\ell
      \left[
        m_j(s_j^{\ell+1}-S_{i,j}^{\ell+1}) - m_i(s_i^{\ell+1}-S_{i,j}^{\ell+1})
      \right]
  = 0
.
   \end{equation}
Combining \eqref{eq:b0_lhs} and \eqref{eq:b0_rhs} gives
    \begin{equation}
b_0^{\ell+1}
  = 
    \left(
      \what{\bfv}_0,\bfeta^{\mathfrak{n}}
    \right)_2
  + 
    \frac{b}{d}
    \left(
      \mathbf{1},\bfs_{\bfW}^{\mathfrak{n}}
    \right)_2
  - 
    \frac{b}{d}
    \left(
      \mathbf{1},\bfs_{\bfW}^{\ell+1}
    \right)_2
.
    \end{equation}
Therefore, with the definition of $\bfs_{\bfW}$ in \eqref{eq:source_term1}, 
    \begin{equation}
    \begin{split}
\left|
  \Delp{(b_0)}
\right| 
  &=
    \frac{b}{d}
    \left|
      \left(
        \mathbf{1},\Delp{(\bfs_{\bfW})}
      \right)_2
    \right|
  \leq
    \frac{b}{d}\sum_{i=1}^N
    \left|
      \;
      \left\|
        \bfw_i^{\ell+1}
      \right\|_2^2
      -
      \left\|
        \bfw_i^\ell
      \right\|_2^2
      \;
    \right|
    \\
  &\leq
    \frac{b}{d}\sum_{i=1}^N
    \left\|
      \Delp{(\bfw_i)}
    \right\|_2
    \left\|
      \bfw_i^{\ell+1}+\bfw_i^\ell
    \right\|_2 
  \leq
    \frac{2b}{d}\rho_{\max}^{\sfrac{1}{2}}u_{\max}^\mathfrak{n}
    \sum_{i=1}^N
    \left\|
      \Delp{(\bfw_i)}
    \right\|_2
    \\
  &\leq
    \frac{2b}{d}\rho_{\max}^{\sfrac{1}{2}}u_{\max}^\mathfrak{n}
    \sqrt{
      N\sum_{i=1}^N
      \left\|
        \Delp{(\bfw_i)}
      \right\|^2_2
    }
  =
    \frac{2b}{d}\rho_{\max}^{\sfrac{1}{2}}u_{\max}^\mathfrak{n}N^{\sfrac{1}{2}} 
    \froNorm{\Delp{\bfW}}
    \\
  &\leq
    \frac{2b}{d}\rho_{\max}^{\sfrac{1}{2}}u_{\max}^\mathfrak{n}N^{\sfrac{1}{2}}
    C_{\bfW}^\mathfrak{n}\Gamma_{\bfW}^\mathfrak{n}
    \left\|
      \Delm{\bfeta}
    \right\|_2
,
    \end{split}
    \end{equation}
where the last line follows from \Cref{theorem:VelocityError}.
    \end{proof}
    
    \subsection{Bounding Source Terms in the Temperature Equation}
    
    \begin{lemma}
    \label{lemma:s_W_Diffs}
The vector $\bfs_{\bfW}^\ell = \left(|\bfw_1^\ell|^2,\cdots,|\bfw_N^\ell|^2\right)^\top$ satisfies the bounds
    \begin{equation}
\|\bfs_{\bfW}^\ell\|_2
  \leq
    (u_{\max}^\mathfrak{n})^2\|P\mathbf{1}\|_2
\qquand
\|\Delp{(\bfs_{\bfW})}\|_2
  \leq
    C_{\bfs}^\mathfrak{n}\Gamma_{\bfW}^\mathfrak{n}
    \|\Delm{\bfeta}\|_2
,
    \end{equation}
where $C_{\bfs}^\mathfrak{n} = 2\rho_{\max}^{\sfrac{1}{2}}u_{\max}^\mathfrak{n}C_{\bfW}^\mathfrak{n}$ and $\Gamma_{\bfW}^\mathfrak{n} = \frac{\dt\veps}{(\veps+\dt z_{\min}^\mathfrak{n})^2}$.
    \end{lemma}

    \begin{proof}
The first inequality follows from the bound on the velocity terms \eqref{eq:u_max_def} in \Cref{prop:GST_properties}: 
    \begin{equation}
\|\bfs_{\bfW}^\ell\|_2^2
  =
    \sum_i[\bfs_{\bfW}]_i^2
  =
    \sum_i\rho_i^2
      \|\bfu_i^\ell\|_2^4
  \leq
    (u_{\max}^\mathfrak{n})^4\sum_i\rho_i^2
.
    \end{equation}
For the second inequality,
    \begin{equation}
    \begin{split}
\|\Delp{(\bfs_{\bfW})}\|_2^2
  &=
    \sum_i
    \left|
      \left(
        \Delp{(\bfw_i)},\bfw_i^{\ell+1}+\bfw_i^\ell
      \right)_2
    \right|^2
  \leq
    \sum_i
    \|\Delp{(\bfw_i)}\|_2^2
\;
    \|\bfw_i^{\ell+1}+\bfw_i^\ell\|_2^2
    \\
  &\leq
    4 \rho_{\max}(u_{\max}^\mathfrak{n})^2\froNorm{\Delp{\bfW}}^2
.
    \end{split}
    \end{equation}
Applying the bound in \Cref{theorem:VelocityError} and taking square roots gives the result.
    \end{proof}

    \begin{lemma}
    \label{lemma:Si_Sij_diffs}
    \begin{subequations}
    \begin{align}
    \label{eq:u_i_minus_u_infty}
\left\|
  \bfu_i^{\ell+1}-\bfu_\infty
\right\|_2
  &\leq
    \frac{\froNorm{\bfW^{\mathfrak{n}}}}{\rho_{\min}^{\sfrac{1}{2}}}
    \;\gamma^\mathfrak{n}
,
    \\
    \label{eq:norm_si_minus_s_ij}
\left|
  s_i^{\ell+1}-S_{i,j}^{\ell+1}
\right|
  &\leq
    \frac
        {4\alpha_{j,i}u_{\max}^\mathfrak{n}\froNorm{\bfW^{\mathfrak{n}}}}
        {\rho_{\min}^{\sfrac{1}{2}}}
    \;\gamma^\mathfrak{n}
,
    \end{align}
    \end{subequations}
where $\gamma^\mathfrak{n}$ is defined in \Cref{lemma:inverse_norms}. 
    \end{lemma}
    
    \begin{proof}
We first prove \eqref{eq:u_i_minus_u_infty}.
From \eqref{eq:W_bound} and the bound in \eqref{eq:A_iterate_bound},
    \begin{equation}
    \label{eq:velocity_norm_0}
\froNorm{\bfW_\mathcal{R}^{\ell+1}}
  =
    \froNorm{\bfA^{\ell+1}}
  \leq
    \gamma^\mathfrak{n}\froNorm{\bfW^{\mathfrak{n}}}
.
    \end{equation}
Then using \Cref{lemma:nullSpaceIter} and the fact that $\bfW_\mathcal{R}^{\ell+1}=\bfW^{\ell+1}-\bfW_\mathcal{N}$,
    \begin{equation}
\froNorm{\bfU^{\ell+1}-\bfU_\infty}^2
  = 
    \froNorm{
        P^{-\sfrac{1}{2}}
        (\bfW^{\ell+1}-\bfW_\mathcal{N})
    }^2
  \leq
    \frac{1}{\rho_{\min}}\froNorm{\bfW_\mathcal{R}^{\ell+1}}^2
  \overset{\eqref{eq:velocity_norm_0}}{\leq}
    \frac{1}{\rho_{\min}}
    (\gamma^\mathfrak{n})^2\froNorm{\bfW^{\mathfrak{n}}}^2
.
      \end{equation}
Therefore,
    \begin{equation}
\|\bfu_i^{\ell+1}-\bfu_\infty\|_2
  \leq
    \froNorm{\bfU^{\ell+1}-\bfU^\infty}
  \leq
    \gamma^\mathfrak{n}
    \frac
        {\froNorm{\bfW^{\mathfrak{n}}}}
        {\rho_{\min}^{\sfrac{1}{2}}}
.
    \end{equation}
We next prove \eqref{eq:norm_si_minus_s_ij}.
Since $\alpha_{i,j}+\alpha_{j,i}=1$,
    \begin{equation}
s_i^{\ell+1}-S_{i,j}^{\ell+1}
  =
    \alpha_{j,i}
    \left(
      \bfu_i^{\ell+1}+\bfu_{i,j}^{\ell+1},\bfu_i^{\ell+1} - \bfu_j^{\ell+1}
    \right)_2
.
    \end{equation}
Thus, using \eqref{eq:u_i_minus_u_infty},
    \begin{equation}
    \begin{split}
\left|s_i^{\ell+1}-S_{i,j}^{\ell+1}\right|
  &\leq
    \max\{\alpha_{i,j}\}
    \left[
      \left\|
        \bfu_i^{\ell+1}
      \right\|_2
      +
      \left\|
        \bfu_{i,j}^{\ell+1}
      \right\|_2
    \right]
    \left\|
      \bfu_i^{\ell+1} - \bfu^{\infty}+\bfu^{\infty} - \bfu_j^{\ell+1}
    \right\|_2
    \\
  &\leq 4\max\{\alpha_{i,j}\}u_{\max}^\mathfrak{n}
    \left(
      \frac{\froNorm{\bfW^{\mathfrak{n}}}}{\rho_{\min}^{\sfrac{1}{2}}}
    \right)
    \gamma^\mathfrak{n}
.
    \end{split}
    \end{equation}
    \end{proof}
    
    \begin{lemma} 
    \label{lemma:s2_Norm}
For all $\ell \geq 1$,
    \begin{align}
\|\bfs^{(2),\ell,\ell-1}\|_2
  \leq C_{\bfs^{(2)}_0}^\mathfrak{n}
    \gamma^\mathfrak{n},
\qwhere
C_{\bfs^{(2)}_0}^\mathfrak{n}
  =
    \frac
        {8\max\{\alpha_{i,j}\}B_{\max}u_{\max}^\mathfrak{n}m_{\max}N^{\sfrac{1}{2}}(N-1)\froNorm{\bfW^{\mathfrak{n}}}}
        {\rho_{\min}^{\sfrac{1}{2}}}
.
    \end{align}
    \end{lemma}

    \begin{proof}
    \begin{equation}
    \begin{split}
\|&\bfs^{(2),\ell,\ell-1}\|_2^2
  =
    \sum_i
    \bigg|
      \sum_{j\neq i}B_{i,j}^\ell
      \big[
        m_j(s_j^{\ell+1}-S_{i,j}^{\ell+1})
        -
        m_i(s_i^{\ell+1}-S_{i,j}^{\ell+1})
      \big]
    \bigg|^2
    \\
  &\leq \sum_{i,j\neq i}\sum_{k\neq i}B_{i,j}B_{i,k}
    \bigg[
      m_i^2
        \big|s_i^{\ell+1}-S_{i,j}^{\ell+1}\big|
        \big|s_i^{\ell+1}-S_{i,k}^{\ell+1}\big| 
      + m_im_j
        \big|s_i^{\ell+1}-S_{i,k}^{\ell+1}\big|
        \big|s_j^{\ell+1}-S_{i,j}^{\ell+1}\big|
    \\
  &\qquad+
      m_im_k
        \big|s_i^{\ell+1}-S_{i,j}^{\ell+1}\big|
        \big|s_k^{\ell+1}-S_{i,k}^{\ell+1}\big| 
      + m_jm_k
        \big|s_j^{\ell+1}-S_{i,j}^{\ell+1}\big|
        \big|s_k^{\ell+1}-S_{i,k}^{\ell+1}\big|
    \bigg]
    \\
  &\leq B_{\max}^2m_{\max}^2 \; 4 \;
    \left(
      \frac
          {4\max\{\alpha_{i,j}\}u_{\max}^\mathfrak{n}\froNorm{\bfW^{\mathfrak{n}}}}
          {\rho_{\min}^{\sfrac{1}{2}}}
      \gamma^\mathfrak{n}
    \right)^2
  N(N-1)^2
.
    \end{split}
    \end{equation}
Combining constants and taking the square root gives the result.
    \end{proof}

    \begin{lemma}
    \label{lemma:s_ell_diffs}
For any $i,j,k\in\{1,\cdots,N\}$,
    \begin{subequations}
    \begin{align}
    \label{eq:sk_diff}
\left|
  s_{k}^{\ell+1}-s_{k}^\ell
\right|
  &\leq
    \frac{2u_{\max}^\mathfrak{n}C_{\bfW}^\mathfrak{n}}{\rho_{\min}^{\sfrac{1}{2}}}\Gamma_{\bfW}^\mathfrak{n}
    \left\|
      \Delm{\bfeta}
    \right\|_2 
,
    \\
\text{and}\qquad
\left|
  S_{i,j}^{\ell+1}-S_{i,j}^\ell
\right|
  &\leq \frac{2u_{\max}^\mathfrak{n}C_{\bfW}^\mathfrak{n}}{\rho_{\min}^{\sfrac{1}{2}}}\Gamma_{\bfW}^\mathfrak{n}
    \left\|
      \Delm{\bfeta}
    \right\|_2
.
    \end{align}
    \end{subequations}
    \end{lemma}

    \begin{proof}
We first prove \eqref{eq:sk_diff}.
As a simple corollary to \Cref{theorem:VelocityError},
    \begin{equation}
    \label{eq:u_k_ell1_ell_diff}
    \begin{split}
\|\bfu_k^{\ell+1}-\bfu_k^\ell\|_2^2
  &\leq
    \sum_{i=1}^N \|\bfu_i^{\ell+1}-\bfu_i^\ell\|_2^2
  =
    \froNorm{\bfU^{\ell+1}-\bfU^\ell}^2
  =
    \froNorm{P^{-\sfrac{1}{2}}\left(\Delp{\bfW}\right)}^2
    \\
  &\overset{\ref{theorem:VelocityError}}{\leq}
    \frac{1}{\rho_{\min}}
    (C_{\bfW}^\mathfrak{n})^2
    (\Gamma_{\bfW}^\mathfrak{n})^2\left\|\Delm{\bfeta}\right\|_2^2
.
    \end{split}
    \end{equation}
Thus,
    \begin{equation}
\left|
  s_k^{\ell+1}-s_k^\ell
\right|
  = 
    \left|
      \left(
        \bfu_k^{\ell+1}-\bfu_k^\ell
        ,
        \bfu_k^{\ell+1}+\bfu_k^\ell
      \right)_2
    \right|
  \leq
    2u_{\max}^\mathfrak{n}
    \|\bfu_k^{\ell+1}-\bfu_k^\ell\|_2
  \leq
    \frac
        {2u_{\max}^\mathfrak{n}C_{\bfW}^\mathfrak{n}}
        {\rho_{\min}^{\sfrac{1}{2}}}
    \Gamma_{\bfW}^\mathfrak{n}
    \|\Delm{\bfeta}\|_2
,
    \end{equation}
as desired.
Next, note that
    \begin{equation}
    \begin{split}
\left|S_{i,j}^{\ell+1}-S_{i,j}^\ell\right|
  &\;\;=\;\;
    \left|
    \left(
        \bfu_{i,j}^{\ell+1}-\bfu_{i,j}^\ell
        ,
        \bfu_{i,j}^{\ell+1}+\bfu_{i,j}^\ell
    \right)_2
    \right|
  \leq
    2u_{\max}^\mathfrak{n}
    \left\|
      \bfu_{i,j}^{\ell+1}-\bfu_{i,j}^\ell
    \right\|_2
    \\
  &\;\;=\;\;
    2u_{\max}^\mathfrak{n}
    \left\|
      \alpha_{i,j}(\bfu_i^{\ell+1}-\bfu_i^\ell)
      +
      \alpha_{j,i}(\bfu_j^{\ell+1}-\bfu_j^\ell)
    \right\|_2
    \\
  &\;\;\leq\;\;
    2u_{\max}^\mathfrak{n}
    \left[
      \alpha_{i,j}
      \left\|
        \bfu_i^{\ell+1}-\bfu_i^\ell
      \right\|_2
      +
      \alpha_{j,i}
      \left\|
        \bfu_j^{\ell+1}-\bfu_j^\ell
      \right\|_2
    \right]
    \\
  &\overset{\eqref{eq:u_k_ell1_ell_diff}}{\leq}
    2u_{\max}^\mathfrak{n}(\alpha_{i,j}+\alpha_{j,i})\frac{C_{\bfW}^\mathfrak{n}}{\rho_{\min}^{\sfrac{1}{2}}}\Gamma_{\bfW}^\mathfrak{n}
    \|\Delm{\bfeta}\|_2
    \\
  &\;\;=\;\;
    \frac
        {2u_{\max}^\mathfrak{n}C_{\bfW}^\mathfrak{n}}
        {\rho_{\min}^{\sfrac{1}{2}}}
    \Gamma_{\bfW}^\mathfrak{n}
    \|\Delm{\bfeta}\|_2
.
    \end{split}
    \end{equation}
    \end{proof}

    \begin{lemma}
    \label{lemma:s2_diffs}
    \begin{equation}
\left\|
  \bfs^{(2),\ell+1,\ell}
  -
  \bfs^{(2),\ell,\ell-1}
\right\|_2
  \leq
    C_{\bfs^{(2)}_1}^\mathfrak{n}\Gamma_{\bfW}^\mathfrak{n}
    \|
      \Delm{\bfeta}
    \|_2
    +
    C_{\bfs^{(2)}_2}^\mathfrak{n}
    \gamma^\mathfrak{n}
    \|\Delm{\bfeta}\|_2
,
    \end{equation}
for $C_{\bfs^{(2)}_1}^\mathfrak{n} = \frac{6B_{\max}u_{\max}^\mathfrak{n}m_{\max}N(N-1)C_{\bfW}^\mathfrak{n}}{\rho_{\min}^{\sfrac{1}{2}}},$ and $C_{\bfs^{(2)}_2}^\mathfrak{n} = C_B^\mathfrak{n} \frac{16\max\{\alpha_{i,j}\}u_{\max}^\mathfrak{n}m_{\max}(N-1)\froNorm{\bfW^{\mathfrak{n}}}}{\rho_{\min}^{\sfrac{1}{2}}}$.
    \end{lemma}

    \begin{proof}
Let $\xi_{i,j}^k = m_j(s_j^k-S_{i,j}^k)-m_i(s_i^k-S_{i,j}^k)$, so that $\left[\bfs^{(2),k_1,k_2}\right]_i = \sum_{j\neq i}B_{i,j}^{k_2}\xi_{i,j}^{k_1}$.
Then using the norm equivalence of the 2- and 1- norms, adding and subtracting $B_{i,j}^\ell\xi_{i,j}^\ell$, and using triangle inequality,
    \begin{equation}
    \label{eq:s2_diff_reference}
    \begin{split}
\left\|
  \bfs^{(2),\ell+1,\ell}
  -
  \bfs^{(2),\ell,\ell-1}
\right\|_2
  &\leq
    \left\|
      \bfs^{(2),\ell+1,\ell}-\bfs^{(2),\ell,\ell-1}
    \right\|_1
  = \sum_{i=1}^N
    \bigg|
      \sum\limits_{\substack{j=1 \\ j\neq i}}^NB_{i,j}^\ell\xi_{i,j}^{\ell+1}-B_{i,j}^{\ell-1}\xi_{i,j}^\ell
    \bigg|
    \\
  &\leq \sum_{i,j\neq i}B_{i,j}^\ell
    \left|
      \xi_{i,j}^{\ell+1}-\xi_{i,j}^\ell
    \right|
    +
    \sum_{i,j\neq i}
    \left|
      B_{i,j}^\ell-B_{i,j}^{\ell-1}
    \right|
    \left|
      \xi_{i,j}^\ell
    \right|
.
    \end{split}
    \end{equation}    
Using \Cref{lemma:s_ell_diffs}, the first term of \eqref{eq:s2_diff_reference} can be bounded as follows:
    \begin{align}
    \nonumber
\sum_{i,j\neq i}&B_{i,j}^\ell
  |\xi_{i,j}^{\ell+1}-\xi_{i,j}^\ell|
  \leq
    B_{\max}^\mathfrak{n}\sum_{i,j\neq i}
    \bigg|
      m_j(s_j^{\ell+1}-S_{i,j}^{\ell+1})
      -
      m_i(s_i^{\ell+1}-S_{i,j}^{\ell+1})
      -
      m_j(s_j^{\ell}-S_{i,j}^{\ell})
      + 
      m_i(s_i^{\ell}-S_{i,j}^{\ell})
    \bigg|
    \\
    \nonumber
  &\leq B_{\max}^\mathfrak{n}m_{\max}\sum_{i,j\neq i}
    \left[
      |s_j^{\ell+1}-s_j^\ell|
      +
      |s_i^{\ell+1}-s_i^\ell|
      +
      |S_{i,j}^{\ell+1}-S_{i,j}^\ell|
    \right]
    \\
    \label{eq:TERM_A}
  &\leq 
    \frac
        {6B_{\max}^\mathfrak{n}u_{\max}^\mathfrak{n}m_{\max}N(N-1)C_{\bfW}^\mathfrak{n}}
        {\rho_{\min}^{\sfrac{1}{2}}}
        \Gamma_{\bfW}^\mathfrak{n}
    \left\|
      \Delm{\bfeta}
    \right\|_2
. 
    \end{align}
Using \Cref{lemma:A_ij_B_ij_Diffs,lemma:Si_Sij_diffs}, the second term of \eqref{eq:s2_diff_reference} can be bounded as follows:
    \begin{equation}
    \label{eq:TERM_B}
    \begin{split}
\sum_{i,j\neq i}
    \left|
      B_{i,j}^\ell-B_{i,j}^{\ell-1}
    \right|
    \left|
      \xi_{i,j}^\ell
    \right|
  &\leq
    C_B^\mathfrak{n}
    \left[
      \gamma^\mathfrak{n}
      \frac
          {8\max\{\alpha_{i,j}\}u_{\max}^\mathfrak{n}m_{\max}\froNorm{\bfW^{\mathfrak{n}}}}
          {\rho_{\min}^{\sfrac{1}{2}}}
    \right]
    \sum_{i,j\neq i}
      \left(
        |\Delm{T_i}| + |\Delm{T_j}|
      \right)
    \\
  &=
    C_B^\mathfrak{n}
    \frac
        {8\max\{\alpha_{i,j}\}m_{\max}u_{\max}^\mathfrak{n}\froNorm{\bfW^n}}
        {\rho_{\min}^{\sfrac{1}{2}}}
    2(N-1)\gamma^\mathfrak{n}
    \|\Delm{\bfeta}\|_2
. 
    \end{split}
    \end{equation}
The result follows upon putting together the terms in \eqref{eq:TERM_A} and \eqref{eq:TERM_B}.
    \end{proof}

    \subsection{Temperature Iterates on Range Space}

    \begin{defn}
    \label{definition:bIteratesVHatMatrix}
Let the columns of $\what{V}=[\what{\bfv}_1,\cdots,\what{\bfv}_{N-1}]\in\bbR^{N\times(N-1)}$ form an orthonormal basis of $\what{\mathcal{R}}$. 
For each $\ell$, let $\bfb^\ell\in\bbR^{N-1}$ be the unique vector of coefficients such that $\bfeta_\mathcal{R}^\ell = \what{V}\bfb^\ell$.
    \end{defn}

    \begin{lemma}
    \label{lemma:SimpleNormsTemperature}
Given $\what{V}\in\bbR^{N\times(N-1)}$ as in \Cref{definition:bIteratesVHatMatrix} and any $\bfy\in\bbR^N$,
    \begin{equation}
\|\what{V}\|_2 = 1
    ,\qquad
\|\what{V}^\top\|_2\leq1
    ,\qquad
\|\what{V}^\top\bfy\|_2\leq \|\bfy\|_2
.
    \end{equation}
    \end{lemma}

    \begin{lemma}
    \label{lemma:bIterates_bDiffs}
The iterates $\bfb^{\ell+1}$ satisfy
    \begin{equation}
    \label{eq:b_iterates}
\bfb^{\ell+1} 
  =
    \what{M}^\ell
    \left[
      \what{V}^\top\bfeta^{\mathfrak{n}}
      +
      \frac{1}{d}\what{V}^\top Q^{-\sfrac{1}{2}}
      \left(
        \bfs_{\bfW}^{\mathfrak{n}}
        -
        \bfs_{\bfW}^{\ell+1}
      \right)
      +
      \frac{\dt}{\veps d}\what{V}^\top Q^{-\sfrac{1}{2}}\bfs^{(2),\ell+1,\ell}
    \right]
,
    \end{equation}
where $\what{M}^\ell = \left(I+\frac{\dt}{\veps}\what{V}^\top\what{Z}^\ell\what{V}\right)^{-1}$.
Therefore
    \begin{equation}
    \label{eq:CauchyErrors}
    \begin{split}
\left\|\Delp{\bfb}\right\|_2
  &\leq
    \left\|
      \Delm{\what{M}}\what{V}^\top
      \left[
        \bfeta^{\mathfrak{n}}
        +
        \frac{1}{d}Q^{-\sfrac{1}{2}}\bfs_{\bfW}^{\mathfrak{n}}
      \right]
    \right\|_2
    \\
  &\quad
    + 
    \frac{1}{d}\left\|
      \what{M}^\ell\what{V}^\top Q^{-\sfrac{1}{2}}\bfs_{\bfW}^{\ell+1}
      -
      \what{M}^{\ell-1}\what{V}^\top Q^{-\sfrac{1}{2}}\bfs_{\bfW}^\ell
    \right\|_2
    \\
  &\quad+ \frac{\dt}{\veps d}
    \left\|
      \what{M}^{\ell}\what{V}^\top Q^{-\sfrac{1}{2}}\bfs^{(2),\ell+1,\ell} 
      -
      \what{M}^{\ell-1}\what{V}^\top Q^{-\sfrac{1}{2}}\bfs^{(2),\ell,\ell-1}
    \right\|_2
.
    \end{split}
    \end{equation}
    \end{lemma}

    \begin{proof}
Write $\bfeta^{\ell+1} = \bfeta_\mathcal{R}^{\ell+1}+\bfeta_\mathcal{N}^{\ell+1} = \what{V}\bfb^{\ell+1}+\bfeta_\mathcal{N}^{\ell+1}$, multiply \eqref{eq:ETAUpdateIterative} by $\what{V}^\top$ to cancel the null space component $\bfeta_\mathcal{N}^{\ell+1}$, then multiply by $\what{M}$ to obtain \eqref{eq:b_iterates}.
To obtain \eqref{eq:CauchyErrors}, take the norm of the difference of consecutive iterates, and apply the triangle inequality.
    \end{proof}

The remainder of this subsection is devoted to bounding the terms in \eqref{eq:CauchyErrors}. 

    \begin{lemma}
    \label{lemma:MHat}
Let $\what{\gamma}^\mathfrak{n} = \frac{\veps}{\veps + \dt\what{z}_{\min}^\mathfrak{n}}$, where $\what{z}_{\min}^\mathfrak{n}>0$ is a lower bound on the positive eigenvalues of $\what{Z}^\ell$, defined in \Cref{prop:z_min_defs}.
Then the matrices $\what{M}^\ell$ and $\Delm{\what{M}}$ satisfy the bounds
    \begin{subequations}
    \begin{align}
    \label{eq:MHat_Norm}
\|\what{M}^\ell\|_2
  &\leq
    \what{\gamma}^\mathfrak{n}
,
    \\
    \label{eq:MHat_Diffs}
\|\Delm{(\what{M}})\|_2
  &\leq
    \Gamma_{\bfeta}^\mathfrak{n}\|\Delm{\what{Z}}\|_2
  \leq
    C_{\what{Z}}^\mathfrak{n}\Gamma_{\bfeta}^\mathfrak{n}\|\Delm{\bfeta}\|_2
,
    \end{align}
    \end{subequations}  
where
    \begin{equation}
C_{\what{Z}}^\mathfrak{n}
  =
    \frac
        {2\max_{i,j}\{c_{i,j}^B\}m_{\max}^{\sfrac{1}{2}}N^{\sfrac{1}{2}}(N^{\sfrac{1}{2}}+1)}
        {\sqrt{2}m_{\min}n_{\min}^{\sfrac{3}{2}}\sqrt{T_{\min}^\mathfrak{n}}}
\qquand
\Gamma_{\bfeta}^\mathfrak{n}
  =
    \frac{\dt}{\veps}\left(\what{\gamma}^\mathfrak{n}\right)^2
.
    \end{equation}
    \end{lemma}

    \begin{proof}
The bound in \eqref{eq:MHat_Norm} is verified in \Cref{appendix:lemma:InverseMatrixNorm}.
To prove \eqref{eq:MHat_Diffs}, use \eqref{eq:MHat_Norm}, the matrix identity $X^{-1}-Y^{-1}=X^{-1}(Y-X)Y^{-1}$, \Cref{lemma:SimpleNormsTemperature}, and \Cref{appendix:Z_and_ZHat_diffs}.
    \end{proof}

    \begin{lemma}
    \label{lemma:TERM1}
The first term of \eqref{eq:CauchyErrors} satisfies
    \begin{equation}
\left\|
  \Delm{\what{M}}
  \left[
    \bfeta^{\mathfrak{n}} + \frac{1}{d}Q^{-\sfrac{1}{2}}\bfs_{\bfW}^{\mathfrak{n}}
  \right]
\right\|_2 
  \leq
    C_{1,1}^\mathfrak{n}\Gamma_{\bfeta}^\mathfrak{n}
    \left\|
      \Delm{\bfeta}
    \right\|_2
,
    \end{equation}
where $C_{1,1}^\mathfrak{n} = \left\| \bfeta^{\mathfrak{n}} + \frac{1}{d}Q^{-\sfrac{1}{2}}\bfs_{\bfW}^{\mathfrak{n}} \right\|_2 C_{\what{Z}}^\mathfrak{n}$.
    \end{lemma}

    \begin{proof}
Use \Cref{lemma:MHat}.
    \end{proof}

    \begin{lemma}
    \label{lemma:TERM2}
The second term of \eqref{eq:CauchyErrors} satisfies
    \begin{equation}
\frac{1}{d}
\left\|
  \what{M}^\ell\what{V}^\top Q^{-\sfrac{1}{2}}\bfs_{\bfW}^{\ell+1}
  -
  \what{M}^{\ell-1}\what{V}^\top Q^{-\sfrac{1}{2}}\bfs_{\bfW}^\ell
\right\|_2
  \leq
    C_{2,1}^\mathfrak{n}\what{\gamma}^\mathfrak{n}\Gamma_{\bfW}^\mathfrak{n}
    \left\|
      \Delm{\bfeta}
    \right\|_2
    +
    C_{2,2}^\mathfrak{n}\Gamma_{\bfeta}^\mathfrak{n}
    \left\|
      \Delm{\bfeta}
    \right\|_2
,
    \end{equation}
where $C_{2,1}^\mathfrak{n} = \frac{C_{\bfs}^\mathfrak{n}}{dn_{\min}^{\sfrac{1}{2}}}$ (with $C_{\bfs}^\mathfrak{n}$ defined in \Cref{lemma:s_W_Diffs}) and $C_{2,2}^\mathfrak{n} = \frac{(u_{\max}^\mathfrak{n})^2\|P\mathbf{1}\|_2C_{\what{Z}}^\mathfrak{n}}{dn_{\min}^{\sfrac{1}{2}}}$.
    \end{lemma}

    \begin{proof}
Add and subtract $\what{M}^\ell\what{V}^\top Q^{-\sfrac{1}{2}}\bfs_{\bfW}^\ell$, factor, use standard inequalities and apply \Cref{lemma:MHat,lemma:s_W_Diffs}:
    \begin{equation}
    \begin{split}
\frac{1}{d}
  &
  \left\|
    \what{M}^\ell\what{V}^\top Q^{-\sfrac{1}{2}}\bfs_{\bfW}^{\ell+1}
    -
    \what{M}^{\ell-1}\what{V}^\top Q^{-\sfrac{1}{2}}\bfs_{\bfW}^\ell
  \right\|_2
    \\
&\quad\leq
  \frac{1}{d}
  \left\|
    \what{M}^\ell
  \right\|
  \left\|
    Q^{-\sfrac{1}{2}}
  \right\|_2
  \left\|
    \Delp{(\bfs_{\bfW})}
  \right\|_2
  +
  \frac{1}{d}
  \left\|
    \Delm{\what{M}}
  \right\|_2
  \left\|
    Q^{-\sfrac{1}{2}}
  \right\|_2
  \left\|
    \bfs_{\bfW}^\ell
  \right\|_2 
    \\
&\quad\leq
  \frac{1}{d}
  \what{\gamma}^\mathfrak{n}
  \frac{1}{n_{\min}^{\sfrac{1}{2}}}C_{\bfs}^\mathfrak{n}
  \Gamma_{\bfW}^\mathfrak{n}
  \left\|
    \Delm{\bfeta}
  \right\|_2 
  +
  \frac
      {\|P\mathbf{1}\|_2(u_{\max}^\mathfrak{n})^2}
      {dn_{\min}^{\sfrac{1}{2}}}
  C_{\what{Z}}^\mathfrak{n}
  \Gamma_{\bfeta}^\mathfrak{n}
  \left\|
    \Delm{\bfeta}
  \right\|_2
.
    \end{split}
    \end{equation}
    \end{proof}

    \begin{lemma}
    \label{lemma:TERM3}
The third term of \eqref{eq:CauchyErrors} satisfies
    \begin{equation}
    \begin{split}
\frac{\dt}{\veps d}
  &\left\|
    \what{M}^\ell\what{V}^\top Q^{-\sfrac{1}{2}}\bfs^{(2),\ell+1,\ell}
    -
    \what{M}^{\ell-1}\what{V}^\top Q^{-\sfrac{1}{2}}\bfs^{(2),\ell,\ell-1}
  \right\|_2
    \\
  &\leq C_{3,1}^\mathfrak{n}\Gamma_{\bfW}^\mathfrak{n}
    \left\|
      \Delm{\bfeta}
    \right\|_2
    +
    C_{3,2}^\mathfrak{n}(\Gamma_{\bfW}^\mathfrak{n}\Gamma_{\bfeta}^\mathfrak{n})^{\sfrac{1}{2}}
    \left\|
      \Delm{\bfeta}
    \right\|_2
    +
    C_{3,3}^\mathfrak{n}\Gamma_{\bfeta}^\mathfrak{n}
    \left\|
      \Delm{\bfeta}
    \right\|_2
,
    \end{split}
    \end{equation}
where
    \begin{equation}
C_{3,1}^\mathfrak{n}
  =
    \frac
        {C_{\bfs_1^{(2)}}^\mathfrak{n}}
        {dn_{\min}^{\sfrac{1}{2}}}
    \frac
        {1-\what{\gamma}^\mathfrak{n}}
        {\what{z}^\mathfrak{n}_{\min}}
,\qquad
C_{3,2}^\mathfrak{n}
  =
    \frac
        {C_{\bfs_2^{(2)}}^\mathfrak{n}}
        {dn_{\min}^{\sfrac{1}{2}}}
,\qquand
C_{3,3}^\mathfrak{n}
  = \frac{C_{\bfs_0^{(2)}}^\mathfrak{n}C_{\what{Z}}^\mathfrak{n}}{dn_{\min}^{\sfrac{1}{2}}}
  \frac{1-\what{\gamma}^\mathfrak{n}}{\what{z}^\mathfrak{n}_{\min}}
.
    \end{equation}
    \end{lemma}

    \begin{proof}
Add and subtract $\what{M}^\ell\what{V}^\top Q^{-\sfrac{1}{2}}\bfs^{(2),\ell,\ell-1}$, and factor:
    \begin{equation}
    \label{eq:TERM3_reference}
    \begin{split}
&\frac{\dt}{\veps d}
\left\|
  \what{M}^\ell\what{V}^\top Q^{-\sfrac{1}{2}}\bfs^{(2),\ell+1,\ell}
  -
  \what{M}^{\ell-1}\what{V}^\top Q^{-\sfrac{1}{2}}\bfs^{(2),\ell,\ell-1}
\right\|_2
    \\
&\leq
  \frac{\dt}{\veps d}
  \left\|
    \what{M}^\ell\what{V}^\top Q^{-\sfrac{1}{2}}
    \left(
      \bfs^{(2),\ell+1,\ell} - \bfs^{(2),\ell,\ell-1}
    \right)
  \right\|_2
  +
  \frac{\dt}{\veps d}
  \left\|
  \Delm(\what{M})\,
    \what{V}^\top Q^{-\sfrac{1}{2}}\bfs^{(2),\ell,\ell-1}
  \right\|_2
    \end{split}
    \end{equation} 
The first term of \eqref{eq:TERM3_reference} can be bounded using \Cref{lemma:s2_diffs,lemma:MHat}:
    \begin{equation}
    \label{eq:s2_term3_a}
    \begin{split}
\frac{\dt}{\veps d}
  \left\|
    \what{M}^\ell
  \right\|_2
  \left\|
    Q^{-\sfrac{1}{2}}
  \right\|_2
  &\left\|
    \bfs^{(2),\ell+1,\ell} - \bfs^{(2),\ell,\ell-1}
  \right\|_2
    \\
  &\leq \frac{1}{dn_{\min}^{\sfrac{1}{2}}}
  \frac{\dt}{\veps}\what{\gamma}^\mathfrak{n}
  \left[
    C_{\bfs^{(2)}_1}^\mathfrak{n}\Gamma_{\bfW}
    \left\|
      \Delm{\bfeta}
    \right\|_2
    +
    C_{\bfs^{(2)}_2}^\mathfrak{n}\what{\gamma}^\mathfrak{n}
    \left\|
      \Delm{\bfeta}
    \right\|_2
  \right]
    \\
  &= 
    \frac
        {C_{\bfs^{(2)}_1}^\mathfrak{n}}
        {dn_{\min}^{\sfrac{1}{2}}}
    \frac
        {1-\what{\gamma}^\mathfrak{n}}
        {\what{z}_{\min}^\mathfrak{n}}
    \Gamma_{\bfW}
    \left\|
      \Delm{\bfeta}
    \right\|_2
    +
    \frac
        {C_{\bfs^{(2)}_2}^\mathfrak{n}}
        {dn_{\min}^{\sfrac{1}{2}}}
    (\Gamma_{\bfW}\Gamma_{\bfeta})^{\sfrac{1}{2}}
    \left\|
      \Delm{\bfeta}
    \right\|_2
.
    \end{split}
    \end{equation}

The second term of \eqref{eq:TERM3_reference} can be bounded using \Cref{lemma:MHat,lemma:s2_Norm}:
    \begin{equation}
    \label{eq:s2_term3_b}
\frac{\dt}{\veps d}
\left\|
  Q^{-\sfrac{1}{2}}
\right\|_2
\left\|
  \bfs^{(2),\ell,\ell-1}
\right\|_2
\left\|
  \what{M}^\ell - \what{M}^{\ell-1}
\right\|_2
  \leq
    \frac
        {C_{\bfs_0^{(2)}}^\mathfrak{n}}
        {dn_{\min}^{\sfrac{1}{2}}}
    \frac
        {1-\what{\gamma}^\mathfrak{n}}
        {\what{z}_{\min}^\mathfrak{n}}
    C_{\what{Z}}^\mathfrak{n} \Gamma_{\bfeta}^\mathfrak{n}
    \left\|
      \Delm{\bfeta}
    \right\|_2
.
    \end{equation}
Putting together \eqref{eq:s2_term3_a} and \eqref{eq:s2_term3_b}, the result follows.
    \end{proof}

    \subsection{Summary of Temperature Term Analysis}

The results of the above sections are summarized by the following:
    \begin{theorem}
    \label{theorem:TemperatureError}
The temperature iterations satisfy the following bounds:
    \begin{equation}
\|\Delp{\bfeta}\|_2
  \leq 
    \bigg[
        C_0^\mathfrak{n}\Gamma^\mathfrak{n}_{\bfW}
      +
        C_{1,1}^\mathfrak{n}\Gamma^\mathfrak{n}_{\bfeta}
      +
        C_{2,1}^\mathfrak{n}\Gamma^\mathfrak{n}_{\bfW}\what{\gamma}^\mathfrak{n}
      +
        C_{2,2}^\mathfrak{n}\Gamma^\mathfrak{n}_{\bfeta}
      +
        C_{3,1}^\mathfrak{n}\Gamma^\mathfrak{n}_{\bfW}
      +
        C_{3,2}^\mathfrak{n}(\Gamma^\mathfrak{n}_{\bfW}\Gamma^\mathfrak{n}_{\bfeta})^{\sfrac{1}{2}}
      +
        C_{3,3}^\mathfrak{n}\Gamma^\mathfrak{n}_{\bfeta}
    \bigg]
    \|\Delm{\bfeta}\|_2
,
    \end{equation}
where the constants, $C_{(\cdot)}$ can be found using the following table:
    \begin{center}
    \begin{tabular}{|c||c|c|c|c|}
    \hline
        Constant: & 
        $C_0$ & 
        $C_{1,1}$ & 
        $C_{2,1},\; C_{2,2}$ & 
        $C_{3,1},\; C_{3,2},\; C_{3,3}$ \\
    \hline
        Lemma: & 
        \ref{lemma:TemperatureNullError} &
        \ref{lemma:TERM1} &
        \ref{lemma:TERM2} &
        \ref{lemma:TERM3} \\
    \hline
    \end{tabular}
    \end{center}
    \end{theorem}

    \begin{proof}
Since $\|\Delp\bfeta\|_2 \leq \|\Delp\bfeta_\mathcal{N}\|_2 + \|\Delp\bfeta_\mathcal{R}\|_2$, then using \Cref{lemma:TemperatureNullError,lemma:TERM1,lemma:TERM2,lemma:TERM3}, the result follows.
    \end{proof}

    \section{Convergence Criterion and Time Step Selection}
    \label{section:timeStepSelection}

In this section, we use \Cref{theorem:VelocityError} and \ref{theorem:TemperatureError} to establish the time step restriction in \Cref{theorem:mainTheorem_TimeStep}

    \subsection{Preliminaries}

For $\ell\in\mathbb{N}\cup\{0\},$ define the tuple
$\bfX^\ell \coloneqq \left(\bfA^\ell,\bfb^\ell,b_0^\ell\right) \in\bbR^{(N-1)\times d} \times \bbR^{N-1}\times\bbR$ and the norm $\vertiii{\bfX}=\froNorm{\bfA}+\|\bfb\|_2+|b_0|$.

    \begin{theorem}
    \label{theorem:CauchyErrorsX}
The Cauchy errors of the iterates $\bfX^\ell$ satisfy
    \begin{equation}
\vertiii{\Delp{\bfX}}
  \leq
    C_{\bfX}^\mathfrak{n}\Gamma_{\bfX}^\mathfrak{n}(\dt,\veps)
    \vertiii{\Delm{\bfX}}
,
    \end{equation}
where
$\Gamma_{\bfX}^\mathfrak{n} = \frac{\dt\veps}{(\veps+\dt z^\mathfrak{n})^2} ,\, z^\mathfrak{n}=\min\{z_{\min}^\mathfrak{n},\what{z}_{\min}^\mathfrak{n}\}\eqqcolon z$, and 
    \begin{equation}
C_{\bfX}^\mathfrak{n}
  =
    C_{\bfW}^\mathfrak{n}
    + C_0^\mathfrak{n}
    + C_{1,1}^\mathfrak{n}
    + C_{2,1}^\mathfrak{n}
    + C_{2,2}^\mathfrak{n}
    + C_{3,1}^\mathfrak{n}
    + C_{3,2}^\mathfrak{n}
    + C_{3,3}^\mathfrak{n}
.
    \end{equation}
    \end{theorem}

    \begin{proof}
The proof is simply a summary of the results in  \Cref{theorem:VelocityError,theorem:TemperatureError}.
    \end{proof}

    \begin{lemma}
The bound $C_{\bfX}^\mathfrak{n}\Gamma_{\bfX}^\mathfrak{n}\leq r$ is equivalent to $Q_r(\dt) \geq 0$, where
    \begin{equation}
    \label{eq:Q_rDefinition}
Q_{r}(\dt)
  =
    rz^2\dt^ 2 + (2rz - C_{\bfX}^\mathfrak{n})\veps\dt + r\veps^2
.
    \end{equation}
    \end{lemma}
    
    \begin{proof}
The proof is a simple calculation.
    \end{proof}

    \begin{lemma}
If $C_{\bfX}^\mathfrak{n}\Gamma_{\bfX}^\mathfrak{n}\leq r<1,$
then convergence of the method is guaranteed.
    \end{lemma}

    \begin{proof}
By \Cref{theorem:CauchyErrorsX}, if $C_{\bfX}^\mathfrak{n}\Gamma_{\bfX}^\mathfrak{n}<r<1,$ then
    \begin{equation}
\lim_{\ell\to\infty}\vertiii{\Delp{\bfX}}
  \leq
    \vertiii{\bfX^1-\bfX^0}\lim_{\ell\to\infty}r^\ell
  = 0
.
    \end{equation}
    \end{proof}

    \subsection{Time Step Selection}
    \label{subsection:GST_time_step}
    
Denote the zeros of $Q_r$, defined in \eqref{eq:Q_rDefinition}, by
    \begin{equation}
    \label{eq:dt_pm}
\dt_\pm
  =
    \frac
        {C_{\bfX}^\mathfrak{n}-2rz}
        {2rz^2}
    \veps
    \pm
    \frac
        {C_{\bfX}^\mathfrak{n}\sqrt{1-\frac{4rz}{C_{\bfX}^\mathfrak{n}}}}
        {2rz^2}
    \veps
.
    \end{equation}
Since the $y$-intercept of $Q_r$ is positive, there are three possible cases:
(1) $\dt_{\pm}$ are complex conjugates, which implies that, $Q_r \geq 0$ for any $\dt>0$;
(2) $\dt_{\pm}$ are both real-valued and negative, so that $Q_r > 0$ for any $\dt > 0$;
(3) $\dt_{\pm}$ are both real-valued and positive, in which case, selecting $\dt$ in the set $S_0 \coloneqq (0,\dt_-] \cup [\dt_+, \infty)$ ensures that $Q_r(\dt)\geq 0,$ and thus, the method will converge.

    \begin{theorem}
Given $\dt_0>0$, there exists $\dt_1$ such that (i) $C\dt_0\leq\dt_1\leq\dt_0$ for a constant $C \in (0,1]$ that is independent of $\veps$, and (ii) for this $\dt_1$,
    \begin{equation}
C_{\bfW}^\mathfrak{n}\Gamma_{\bfW}^\mathfrak{n}
  + C_0^\mathfrak{n}\Gamma_{\bfW}^\mathfrak{n}
  + C_1^\mathfrak{n}\Gamma_{\bfX}^\mathfrak{n}
  <1
,
    \end{equation}
thus ensuring convergence of the GST method.
    \end{theorem}
    
    \begin{proof}
The only nontrivial case to consider is when $\dt_{\pm}$ are both real-valued and positive.
If $\dt_0 \leq \dt_-$ or $\dt_+ \leq \dt_0$, then the result follows trivially with $\dt_0 = \dt_1$, i.e., $C=1$.
If $\dt_- < \dt_0 < \dt_+$, then setting $\dt_1 = \dt_-$ gives, via \eqref{eq:dt_pm},
    \begin{equation}
\frac
    {\dt_1}
    {\dt_0}
  >
    \frac
        {\dt_-}
        {\dt_+}
  =
    \frac
        {C_{\bfX}^\mathfrak{n}-2rz - C_{\bfX}^\mathfrak{n}\sqrt{1-\frac{4rz}{C_{\bfX}^\mathfrak{n}}}}
        {C_{\bfX}^\mathfrak{n}-2rz + C_{\bfX}^\mathfrak{n}\sqrt{1-\frac{4rz}{C_{\bfX}^\mathfrak{n}}}}
  \eqqcolon
    C \in (0,1]
,
    \end{equation}    
and $C$ is independent of $\veps$.
    \end{proof}

    \section{Numerical demonstration}
    \label{section:numerics}
    
To demonstrate the application of the GST method, we consider a slab geometry, in which the distribution functions $f_i$ are invariant in the $x_2$ and $x_3$ direction.
In this setting, \eqref{eq:BGKequation} takes the form
    \begin{equation}
    \label{eq:BGKequation_slab} 
\frac{\partial f_i}{\partial t}
+
v_1 \partial_{x}f_i
  =
    \frac{1}{\veps}\sum_j\lambda_{i,j}(M_{i,j}-f_i),
\quad \forall i\in\{1,\cdots,N\}
,
    \end{equation}
where $[\bfu_{i,j}]_2 = [\bfu_{i,j}]_3=0$, and to simplify the notation, we let $x_1 = x$, $[\bfu_i]_1 = u_i$, and $[\bfu_{i,j}]_1=u_{i,j}$.
Rather than solve \eqref{eq:BGKequation_slab} in the full three-dimensional velocity space, we apply a multi-species extension of Chu reduction \cite{ChuReduction}.
For each $i \in \{1,\cdots ,N\}$, let
    \begin{subequations}
    \begin{align}
g_i(x,v_1,t)
  &=
    \int_{\bbR^2}f_i(x,\bfv,t)\,\diff v_2\,\diff v_3
,
    \\
h_i(x,v_1,t) 
  &=
    \int_{\bbR^2}(v_2^2+v_3^2)f_i(x,\bfv,t)\,\diff v_2\,\diff v_3
.
    \end{align}
    \end{subequations}
Then \eqref{eq:BGKequation_slab} can be reduced to a system of two equations for $g_i$ and $h_i$:
    \begin{subequations}
    \label{eq:gh_system}
    \begin{align}
\frac{\partial g_i}{\partial t}
+
v_1\frac{\partial g_i}{\partial x}
  &=
    \frac{1}{\veps}\sum_j\lambda_{i,j}(G_{i,j}-g_i)
,
    \\
\frac{\partial h_i}{\partial t}
+
v_1\frac{\partial h_i}{\partial x}
  &=
    \frac{1}{\veps}\sum_j\lambda_{i,j}(H_{i,j}-h_i)
,
    \end{align}
    \end{subequations}
where $G_{ij}(v_1) = M_{n_i,u_{i,j},T_{i,j}/m_i}^{(1)}(v_1)$ is a Maxwellian (refer to \Cref{defn:Maxwellian}) with $d=1$, $H_{i,j} = 2\frac{T_{i,j}}{m_i}G_{i,j}$, and
    \begin{equation}
n_i 
  =
    \int_\bbR g_i\,\diff v_1
,\qquad
n_i u_i
  =
    \int_\bbR v_1g_i\,\diff v_1
,\qquad
\frac{3n_i}{m_i}T_i
  =
    \int_\bbR (v_1-u_i)^2g_i\,\diff v_1
    +
    \int_\bbR h_i\,\diff v_1
.
    \end{equation}

We use the IMEX approach outlined in \Cref{section:imex} to solve \eqref{eq:gh_system}.
The Butcher tableaux for the IMEX method is taken from Section 2.6 of \cite{ascher1997implicit}.
The implicit step takes the form of a backward Euler method and the associated moment updates for $n_i$, $u_i$, and $T_i$ take exactly the form of \eqref{eq:BEstep}.
Thus, the GST method still forms the key component of the implicit step, and all of the analysis above carries through without any modification.
The spatial discretization of the advection operator uses a second-order finite volume scheme with minmod-type slope limiters.
The velocity is discretized using standard discrete velocity models based on evenly spaced points and a mid-point quadrature rule.
Details of these discretizations can be found in \cite{habbershaw2022progress}.
The strategy for selecting the time step is to use the natural hyperbolic restriction of the advection operators in the equations.
In practice, we observe that the GST method converges without the restrictions set in \Cref{subsection:GST_time_step}.
However, for general problems, these restrictions provide a rigorous fall-back whenever the GST method fails to converge at the hyperbolic time step.

    \subsection{Sod problem with two identical species}
    \label{section:sodTest}

To verify the multi-species code, we perform a simulation of the Sod shock tube problem \cite{sod1978survey}, a Riemann problem for the compressible Euler equations of a perfect gas.
The initial data are $(n_{\textnormal{L}},\bfu_{\textnormal{L}},T_{\textnormal{L}})=(1,\mathbf{0},1)$ for $x \leq 0$, and $(n_{\textnormal{R}},\bfu_{\textnormal{R}},T_{\textnormal{R}})=(0.125,\mathbf{0},0.8)$ for $x \geq 0$.
An analytical solution for the Riemann problem can be found in \cite{toro2013riemann}.
For small values of $\veps$, both the single species BGK equation \cite{coron} and multi-species BGK (M-BGK) equations \cite{haack2017model} with properly initialized data formally recover the Sod solution.
Because the slab problem \eqref{eq:BGKequation_slab} is formulated in three velocity dimensions, the adiabatic index for the limiting Euler equations is $\gamma = 5/3$.

In the current test, we treat the left- and right-hand side of the domain as different particle species, each having mass $m_i=1$.
The kinetic initial data is
    \begin{equation}
    \label{eq:Sod}
f_1(x,\bfv,0)
  =
    \begin{cases}
        M_{n_{\textnormal{L}},\bfu_{\textnormal{L}},T_{\textnormal{L}}}^{(3)} & x \leq 0
        \\
        M_{0.001,\bfu_{\textnormal{R}},T_{\textnormal{R}}}^{(3)} & x > 0
    \end{cases}
\qquand
f_2(x,\bfv,0)
  =
    \begin{cases}
        M_{0.001,\bfu_{\textnormal{L}},T_{\textnormal{L}}}^{(3)} & x \leq 0
        \\
        M_{n_{\textnormal{R}},\bfu_{\textnormal{R}},T_{\textnormal{R}}}^{(3)} & x > 0
    \end{cases}
,
    \end{equation}
where $M_{n,\bfu,\theta}^{(d)}$ denotes the $d$-dimensional Maxwellian, defined in \eqref{eq:Maxwellian_def}.
Trace densities are used in \eqref{eq:Sod} instead of voids in order to avoid negativity in the numerical solutions.

The numerical solution is computed on a discretization of the phase space $X\times V = [-1,1]\times[-10,10]$, using periodic boundary conditions and a mesh with $N_x=256$ by $N_v=192$ cells of uniform size.
The time step selection is driven by the advection in  \eqref{eq:gh_system}: $\dt=0.9\frac{2/Nx}{2\max\{|v_l|\}}$.
\Cref{figure:sodTestSmallEps} shows single species and multi-species simulations of the model with $\veps=10^{-4}$, run to a final time of $t_{\textnormal{F}}=0.2$, compared with the Sod solution for the associated Euler equations.
The results show that both BGK and M-BGK models perform well as numerical solvers for the Sod problem in this highly collisional regime.

\begin{figure}[ht]
    \centering
    \begin{subfigure}{0.32\textwidth}
        \includegraphics[width=\textwidth]{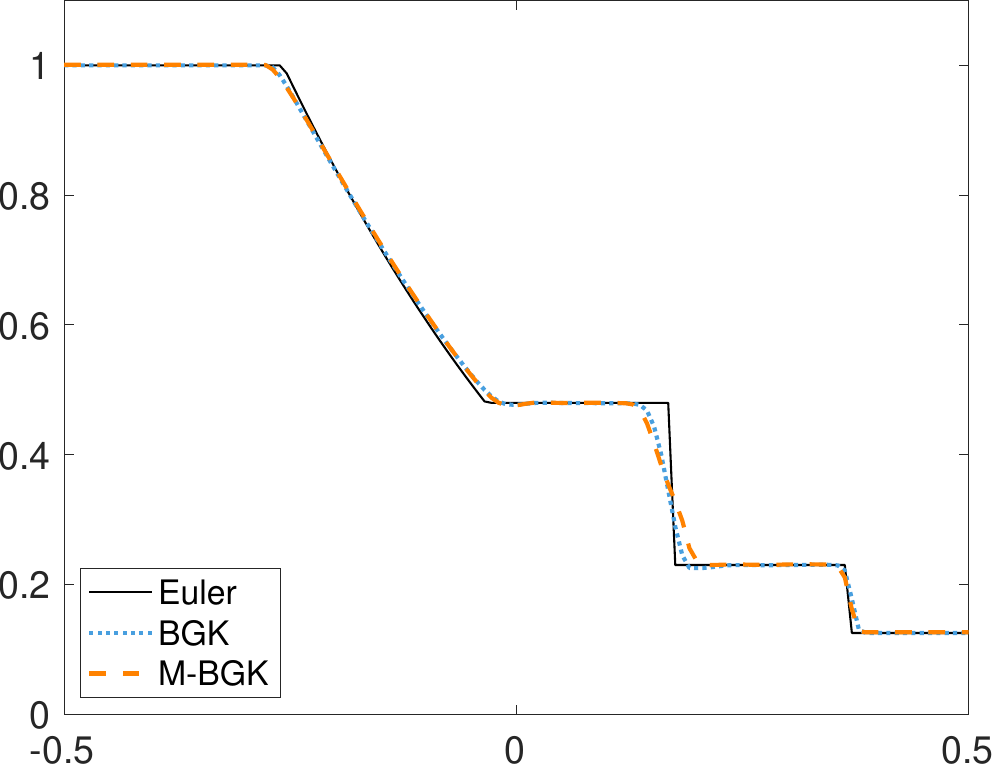}
        \subcaption{Total number density.}
    \end{subfigure}
    \begin{subfigure}{0.32\textwidth}
        \includegraphics[width=\textwidth]{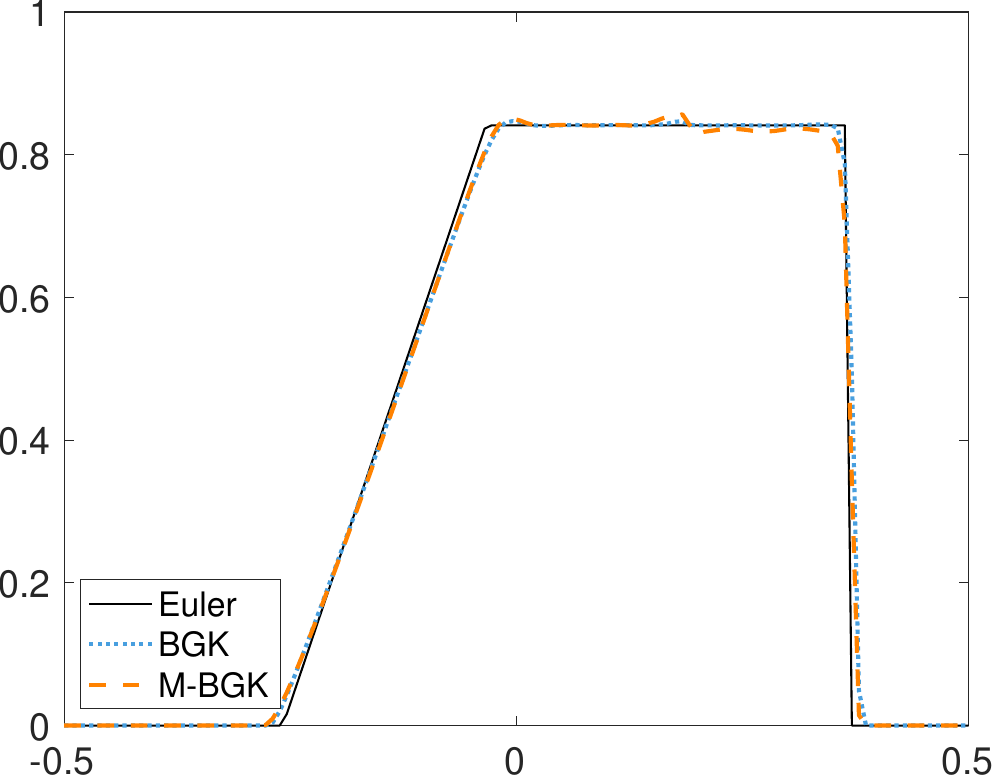}
        \subcaption{Total velocity.}
    \end{subfigure}
    \begin{subfigure}{0.32\textwidth}
        \includegraphics[width=\textwidth]{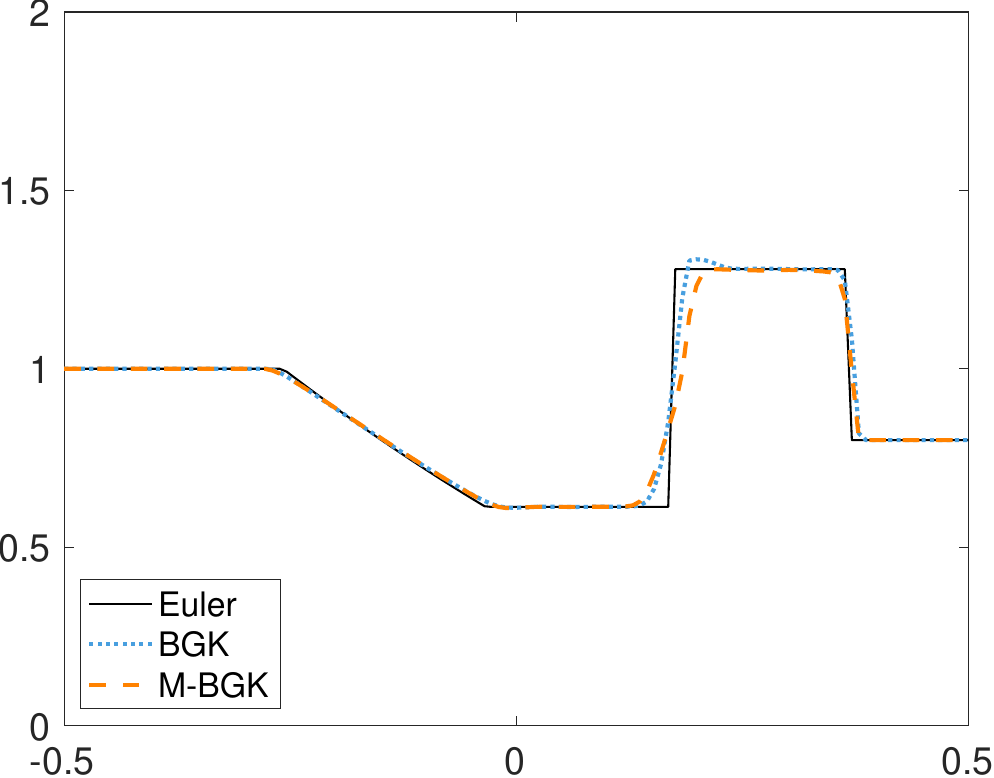}
        \subcaption{Total temperature.}
    \end{subfigure}
    \caption{The total density, total bulk velocity, and total temperature moments are compared to the solution of the Euler equations at time $t=0.2$, with $\veps=10^{-4}$.}
    \label{figure:sodTestSmallEps}
\end{figure}

    \subsection{Three Species Mixture: Ar-Kr-Xe}
The goal of this test is to demonstrate the benefits of the IMEX method when the problem becomes stiff (i.e., as $\varepsilon\to0$).
The following CFL-type conditions on the time step are required for the explicit and IMEX methods (which use a minmod-type slope limiter in the flux computation) to ensure stability and positivity of the numerical solution:
    \begin{equation}
\dt^\mathfrak{n}
    \leq 
  \begin{cases}
\frac{\veps\Delta x}{2\veps\max\{|v_l|\} + 
    \Delta x\Lambda^\mathfrak{n}}
&\!\!\!\!\eqqcolon
  \dt_{\textnormal{EX}}^\mathfrak{n}
    \\
\frac{\Delta x}{2\max\{|v_l|\}}
&\!\!\!\!\eqqcolon
  \dt_{\textnormal{IMEX}}
  \end{cases}
,\qqwhere
\Lambda^\mathfrak{n}
  =
    \max_{i=1}^N
    \left\{
      \max_{\kappa}
      \left\{
        \sum_{j=1}^{N}\lambda_{i,j}^\mathfrak{n}(x_\kappa)
      \right\}
    \right\}
.
    \end{equation}
As $\veps \to 0$, the condition on $\dt_{\textnormal{EX}}^\mathfrak{n}$ becomes restrictive.

To compare the methods, we simulate an interface problem, with a mixture of Argon, Krypton, and Xenon; the masses and diameters of these particle species are taken from \cite{book:MathTheory_NonUniformGases_Chapman_Cowling} and are given below in SI units: 
    \begin{equation}
    \begin{pmatrix}
      m_{\textnormal{Ar}} \\ m_{\textnormal{Kr}} \\ m_{\textnormal{Xe}}
    \end{pmatrix}
  =
    \begin{pmatrix}
      6.6335209 \\ 13.914984 \\ 21.801714
    \end{pmatrix}
\times 10^{-26} \textnormal{ kg}
,\qquad
    \begin{pmatrix}
      d_{\textnormal{Ar}} \\ d_{\textnormal{Kr}} \\ d_{\textnormal{Xe}}
    \end{pmatrix}
  =
    \begin{pmatrix}
      3.659 \\ 4.199 \\ 4.939
    \end{pmatrix}
\times 10^{-10} \textnormal{ m}
.
    \end{equation}
The setup for the test is similar to that of \Cref{section:sodTest}: a mixture of gases at rest is initialized with a high temperature on the left, and a low temperature on the right of the domain.
Specifically, the mixture of Ar, Kr, and Xe is initialized on the left half of the 1-dimensional spatial domain with number densities $(n_{\textnormal{Ar}}^0,n_{\textnormal{Kr}}^0,n_{\textnormal{Xe}}^0) = (5,5,0.5)$ $\times 10^{25}$ m$^{-3}$, and common temperature $T_{\textnormal{L}}=10$; on the right half, the mixture is initialized with number densities $(n_{\textnormal{Ar}}^0,n_{\textnormal{Kr}}^0,n_{\textnormal{Xe}}^0) = (0.5,0.5,5)$ $\times 10^{25}$ m$^{-3}$, and the common temperature $T_{\textnormal{R}}=1$.
The initial kinetic distributions are given by
    \begin{subequations}
    \begin{align}
f_{\textnormal{Ar}}(x,\bfv,0)
  &=
    \chi_{[-1,0]}(x)M_{5,\mathbf{0},10/ m_{\textnormal{Ar}}}^{(3)}(\bfv)
    +
    \chi_{(0,1]}(x)M_{0.5,\mathbf{0},1/m_{\textnormal{Ar}}}^{(3)}(\bfv)
,
    \\
f_{\textnormal{Kr}}(x,\bfv,0)
  &=
    \chi_{[-1,0]}(x)M_{5,\mathbf{0},10/m_{\textnormal{Kr}}}^{(3)}(\bfv)
    +
    \chi_{(0,1]}(x)M_{0.5,\mathbf{0},1/m_{\textnormal{Kr}}}^{(3)}(\bfv)
,
    \\
f_{\textnormal{Xe}}(x,\bfv,0)
  &=
    \chi_{[-1,0]}(x)M_{0.5,\mathbf{0},10/ m_{\textnormal{Xe}}}^{(3)}(\bfv)
    +
    \chi_{(0,1]}(x)M_{5,\mathbf{0},1/m_{\textnormal{Xe}}}^{(3)}(\bfv)
,
    \end{align}
    \end{subequations}
where $\chi_{\mathcal{S}}$ is the indicator function on the set $\mathcal{S}$, and $M_{n,\bfu,\theta}^{(d)}$ denotes the $d$-dimensional Maxwellian, defined in \eqref{eq:Maxwellian_def}.

The numerical solution is computed on a discretization of the phase space $X\times V = [-1,1]\times[-30,30]$, using reflective boundaries; the mesh consists of $N_x=256 \times N_v=192$ cells of uniform size.
The IMEX method (using the GST solver) uses a time step of size $\dt = 0.9\dt_{\textnormal{IMEX}}$, and the fully explicit method uses a time step of size $\dt^\mathfrak{n}=0.9\Delta t_{\textnormal{EX}}^\mathfrak{n}$; the simulation is run to a final time of $t_{\textnormal{F}}=0.1$ $\mu$s.

Moment plots for numerical solutions computed with the IMEX and fully explicit method are compared in \Cref{figure:AKX_smallEps} for $\veps\approx0.001$.
The computed solutions align closely.
However, the explicit method requires $7089$ time steps, while the IMEX method only requires $854$ steps.
In practice, as $\veps$ decreases, the number of time steps for the explicit method grows, slowing the simulation, while the IMEX method experiences no slow down in computation time.

\begin{figure}[ht]
    \centering
    \begin{subfigure}{0.32\textwidth}
        \includegraphics[width=\textwidth, height=0.6 \textwidth]{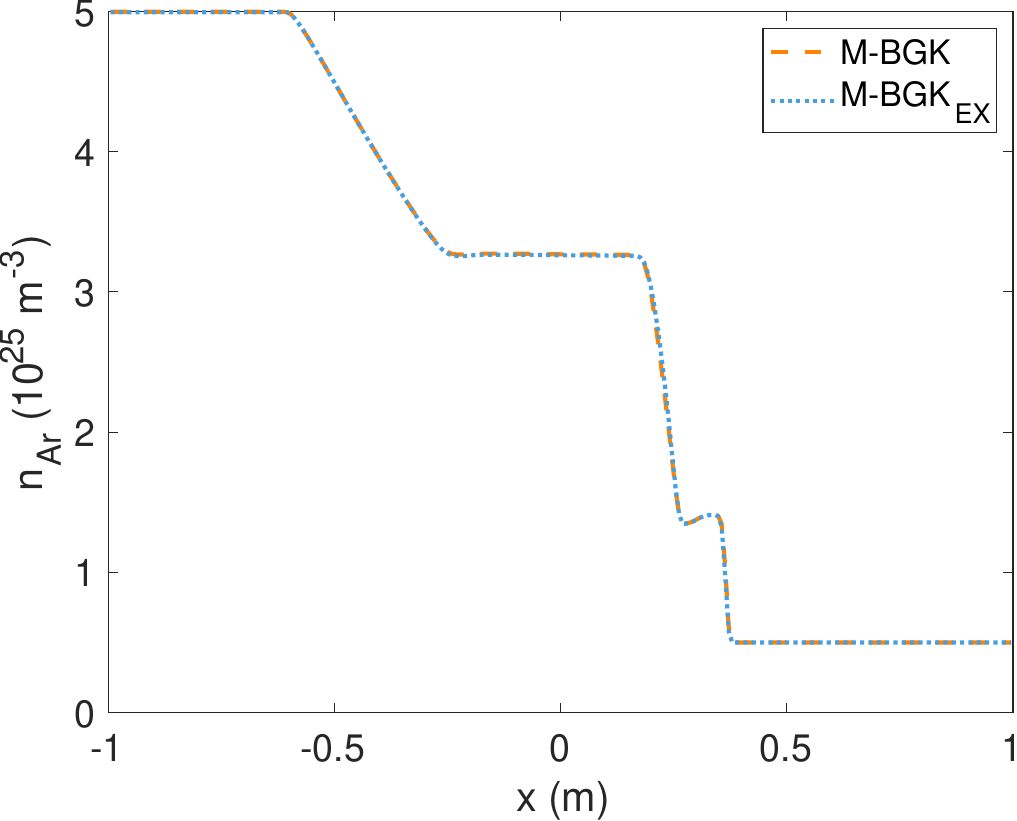}
        \subcaption{Density of Ar}
    \end{subfigure}
    \begin{subfigure}{0.32\textwidth}
        \includegraphics[width=\textwidth, height=0.6 \textwidth]{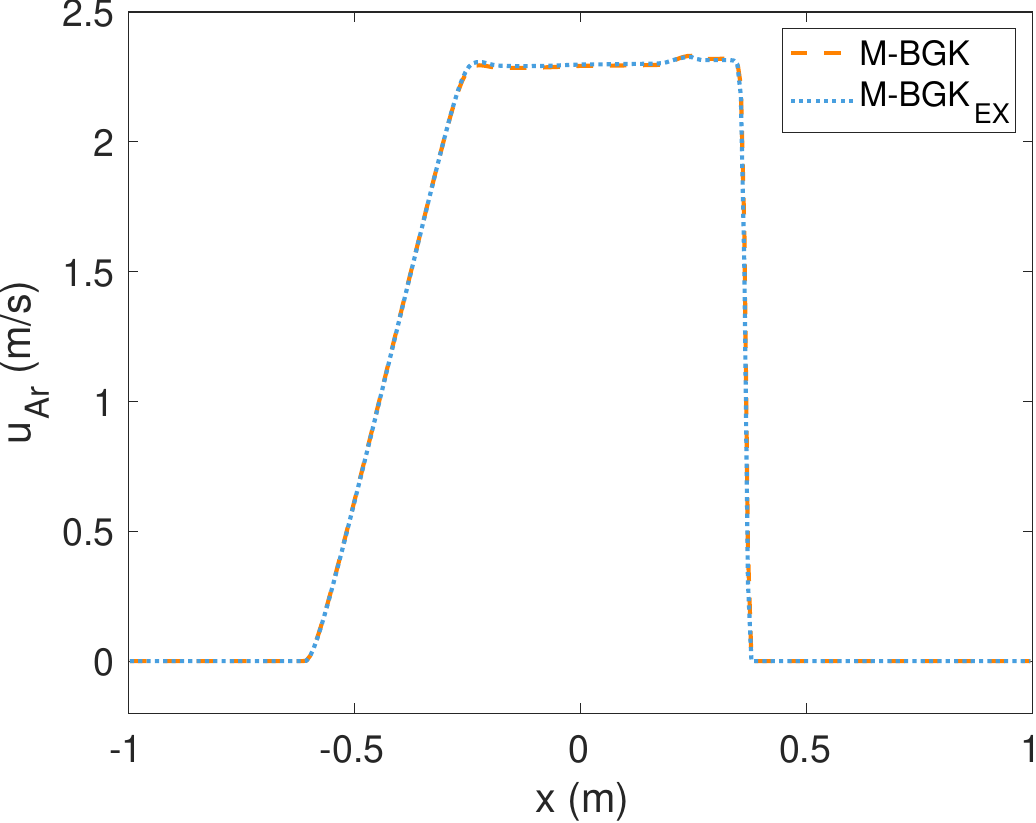}
    \subcaption{Bulk velocity of Ar}
    \end{subfigure}
    \begin{subfigure}{0.32\textwidth}
        \includegraphics[width=\textwidth, height=0.6 \textwidth]{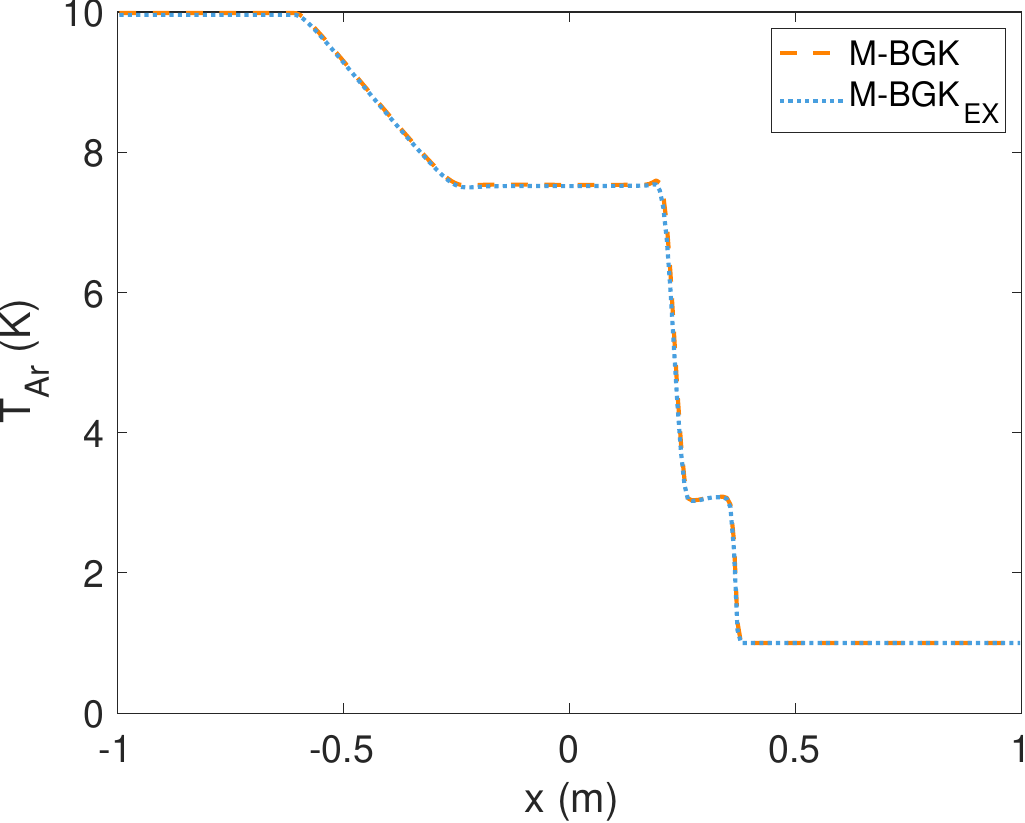}
        \subcaption{Temperature of Ar}
    \end{subfigure}
\\
    \begin{subfigure}{0.32\textwidth}
        \includegraphics[width=\textwidth, height=0.6 \textwidth]{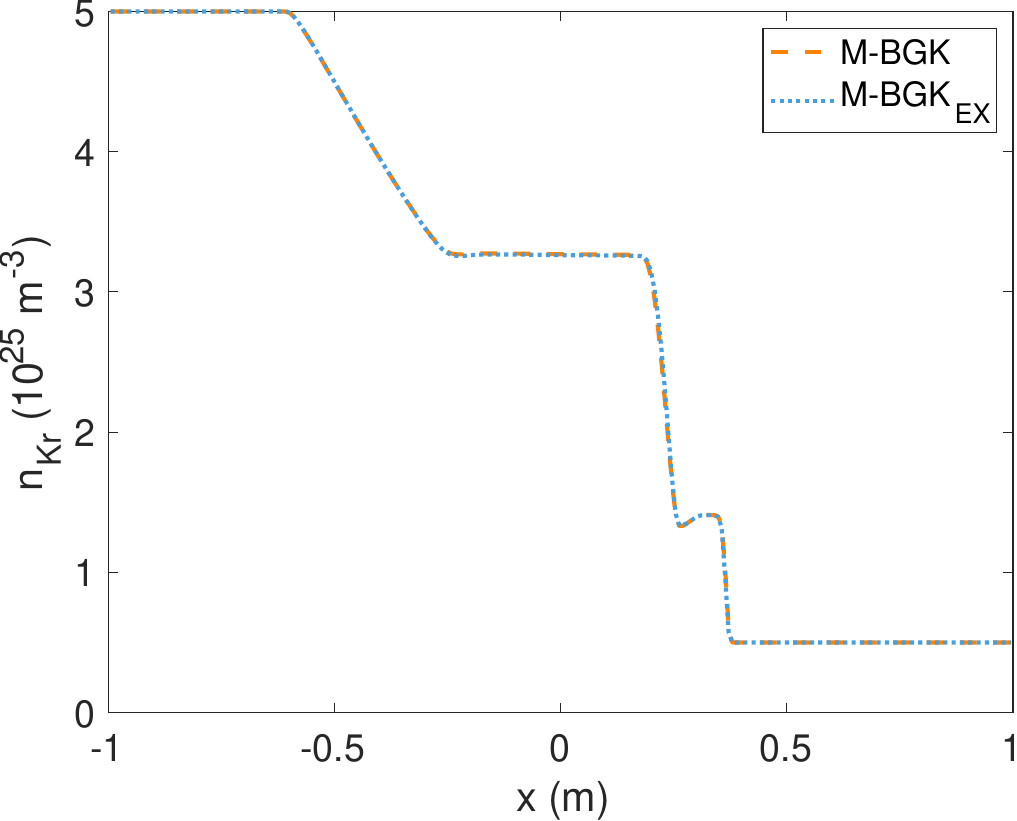}
        \subcaption{Density of Kr}
    \end{subfigure}
    \begin{subfigure}{0.32\textwidth}
        \includegraphics[width=\textwidth, height=0.6 \textwidth]{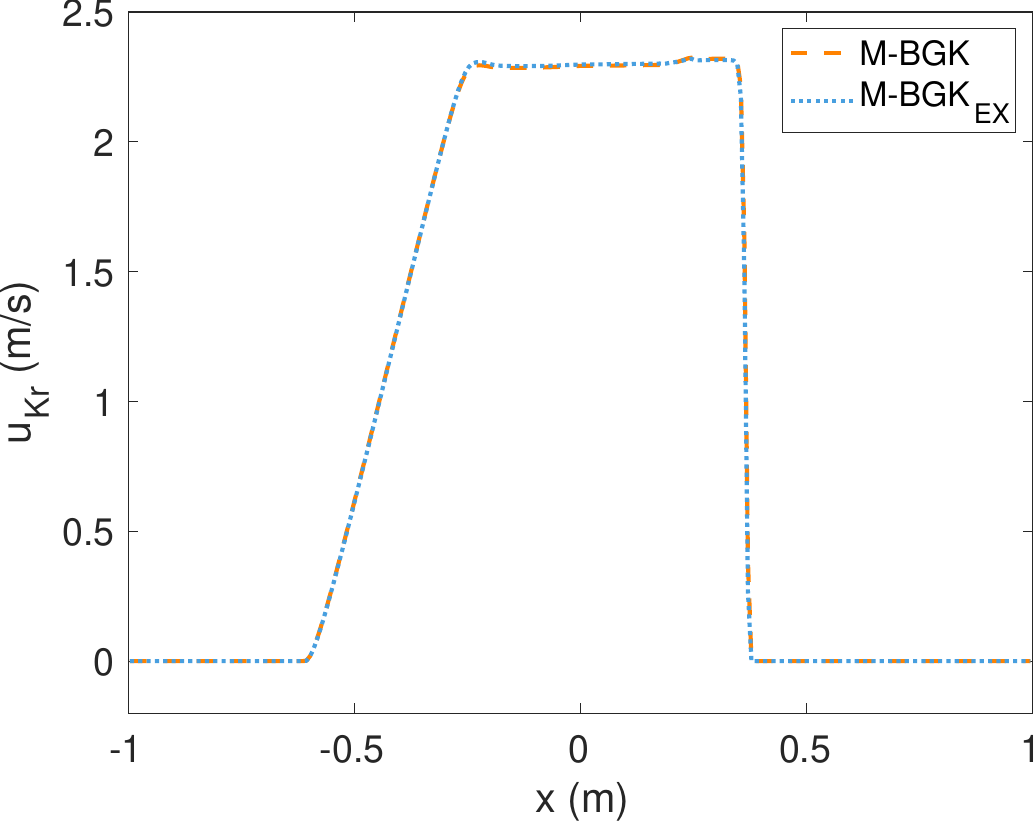}
    \subcaption{Bulk velocity of Kr}
    \end{subfigure}
    \begin{subfigure}{0.32\textwidth}
        \includegraphics[width=\textwidth, height=0.6 \textwidth]{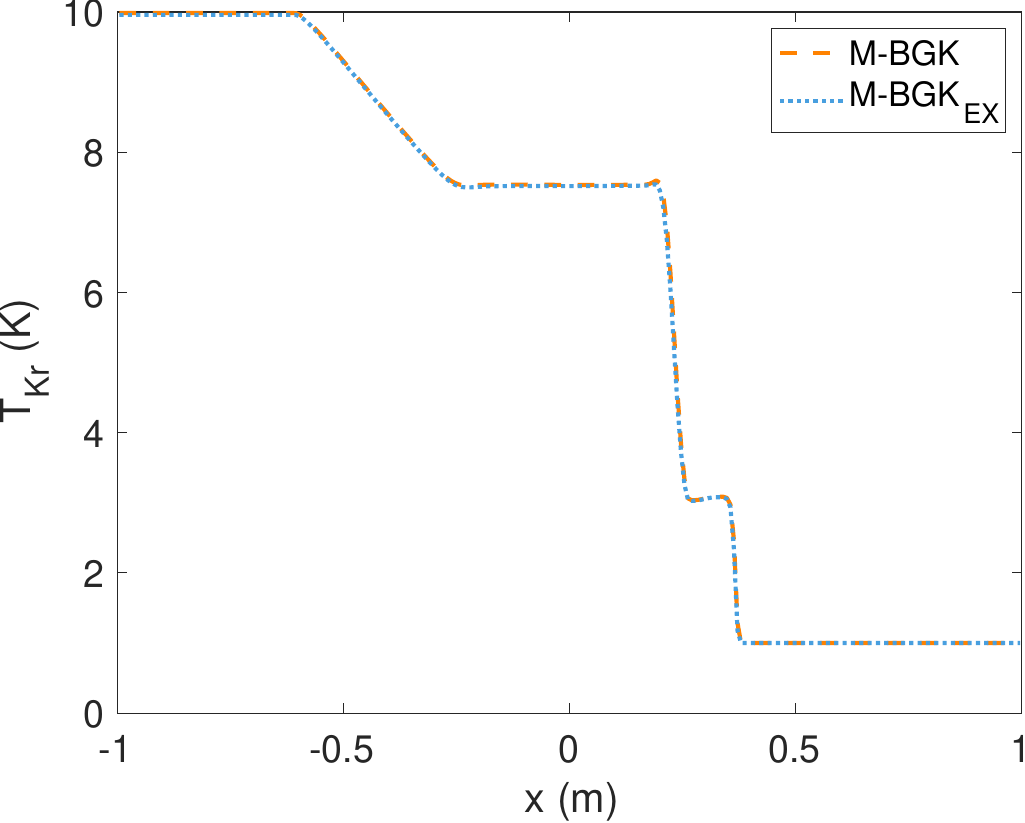}
        \subcaption{Temperature of Kr}
    \end{subfigure}
\\
    \begin{subfigure}{0.32\textwidth}
        \includegraphics[width=\textwidth, height=0.6 \textwidth]{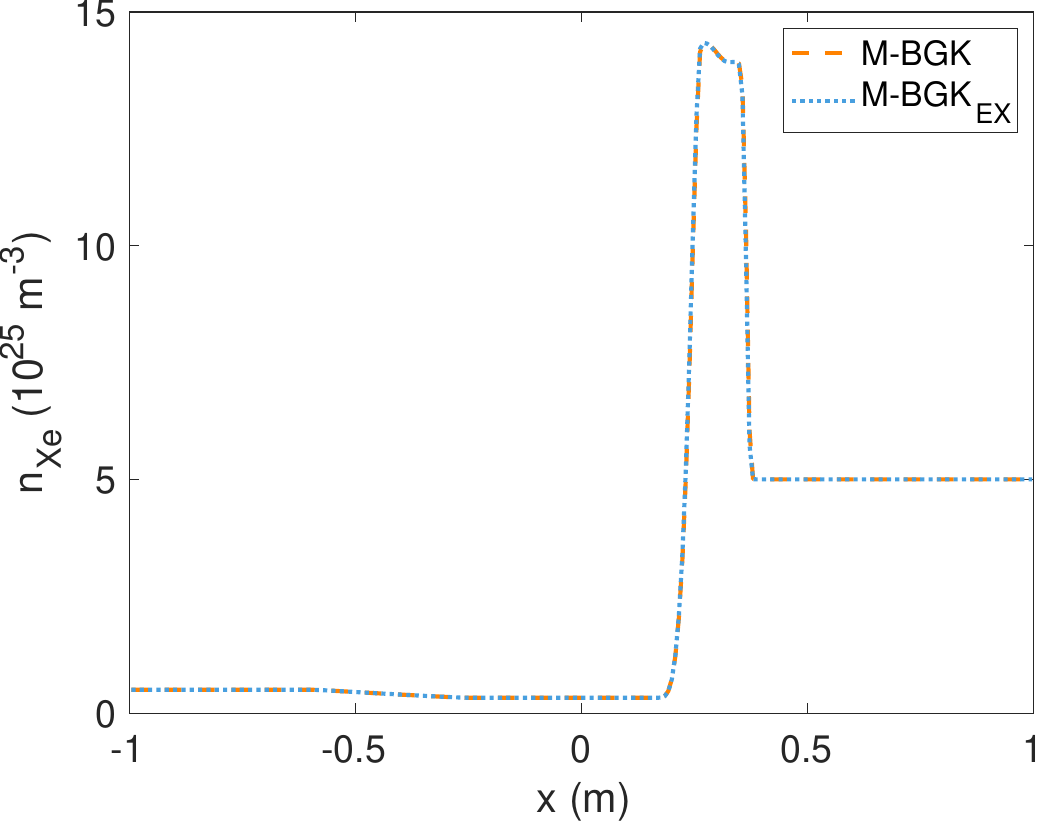}
        \subcaption{Density of Xe}
    \end{subfigure}
    \begin{subfigure}{0.32\textwidth}
        \includegraphics[width=\textwidth, height=0.6 \textwidth]{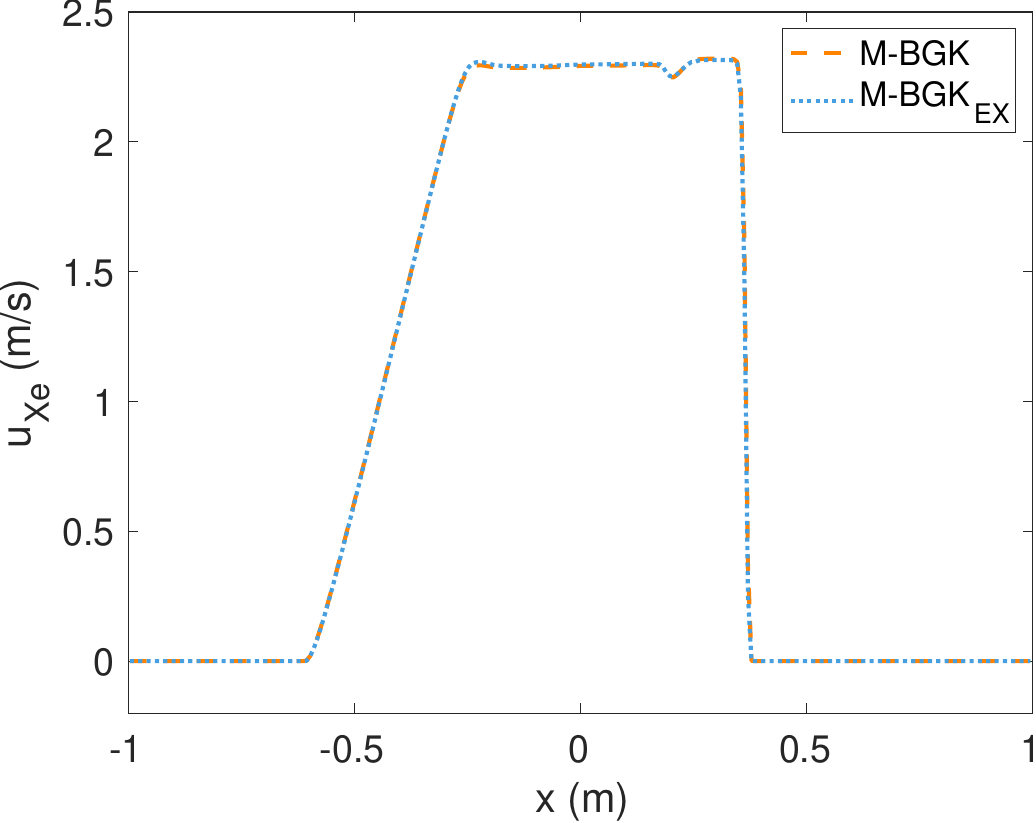}
    \subcaption{Bulk velocity of Xe}
    \end{subfigure}
    \begin{subfigure}{0.32\textwidth}
        \includegraphics[width=\textwidth, height=0.6 \textwidth]{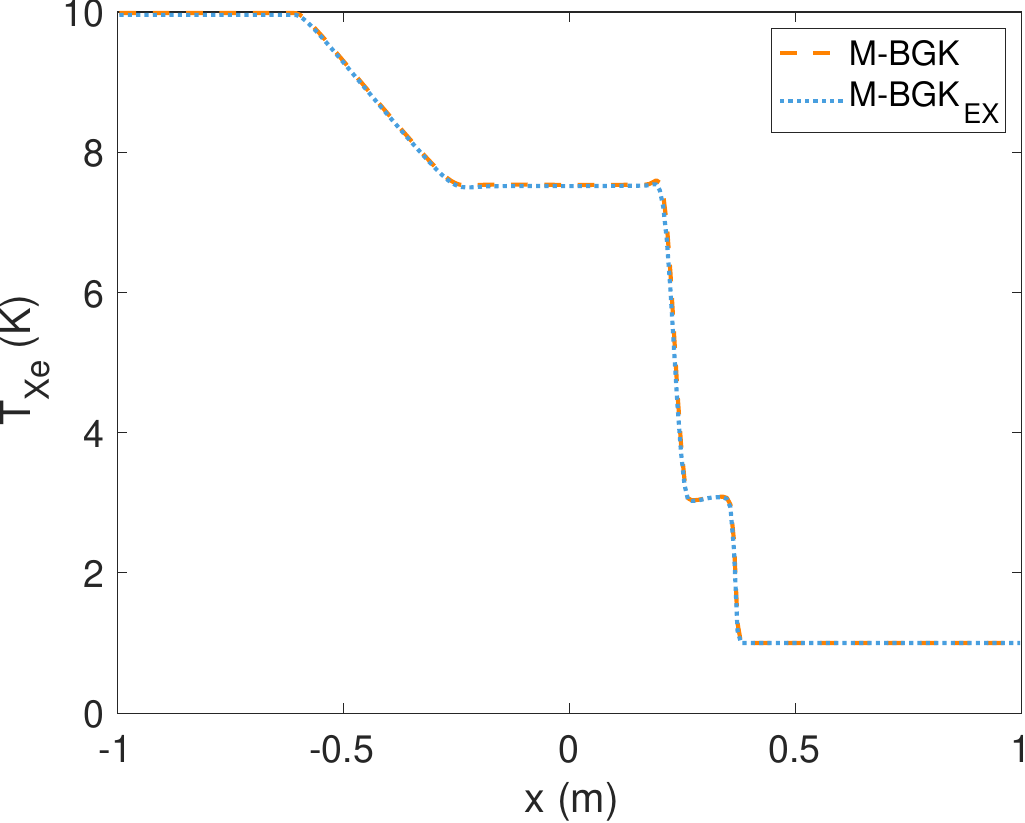}
        \subcaption{Temperature of Xe}
    \end{subfigure}
    \caption{Noble gas mixture of Ar, Kr, and Xe: Moment plots of species number density, bulk velocity, and temperature, are given at $t_{\textnormal{F}}=0.1$ $\mu$s, for a mixture with $\veps \approx 0.00103$. M-BGK uses the GST algorithm for the moment solve. M-BGK$_{\textnormal{EX}}$ uses a fully explicit method.}
    \label{figure:AKX_smallEps}
\end{figure}

    \section{Conclusions}
    \label{section:conclusions}
    
In this paper, we have presented a simple iterative approach for solving a set of implicit moment equations as part of a larger IMEX-based strategy for computing solutions of a multi-species BGK model.
The main challenge is that the collision frequencies in the moment equations themselves depend on the moments in a complicated way.
The key contribution is a rigorous contraction mapping proof, which shows that the iterative method converges to a unique solution under reasonable time steps that do not degenerate to zero when the collision parameter $\veps \to 0$.
This implicit solve is a key component in the IMEX discretization of the multi-species BGK model, the implementation of which is done in slab geometry but closed in three-dimensional velocity space using a Chu reduction technique \cite{ChuReduction}.

In future work, we will examine similar problems for multi-species ellipsoidal statistical models (ES-BGK) \cite{holway1965kinetic,andries2000gaussian} that, unlike standard BGK models, recover in highly collisional regimes, the correct Prandtl number for the  compressible Navier-Stokes equations.  
We will also extend the current code structure used here to generate results for more general problems in multi-dimensional domains.  
We will also explore a Fokker-Planck analog of the BGK model, often referred to as the Lenard-Bernstein model \cite{lenard1958plasma}, that is useful in plasma physics simulations \cite{endeve2022conservative,francisquez2022improved,ulbl2022implementation}.

    \appendix

    \section{Inverse Matrix Norms}
    \label{appendix:lemma:InverseMatrixNorm}
The following lemma is used in the proofs of \Cref{lemma:inverse_norms} and \Cref{lemma:MHat}.

    \begin{lemma}
The matrices $M^\ell$ and $\what{M}^\ell$, defined in \Cref{lemma:W_RangeSpaceIterates} and \Cref{lemma:bIterates_bDiffs}, respectively, satisfy the bounds
    \begin{equation}
\|M^\ell\|_2
  \leq
    \frac{\veps}{\veps+\dt z_{1}^\ell}
  \leq
    \frac{\veps}{\veps+\dt z_{\min}^\mathfrak{n}}
\qquand
\|\what{M}^\ell\|_2
  \leq
    \frac{\veps}{\veps+\dt\what{z}_{1}^\ell}
  \leq
    \frac{\veps}{\veps+\dt\what{z}_{\min}^\mathfrak{n}}
,
    \end{equation}
where $z_1^\ell \geq z_{\min}^\mathfrak{n}$ and $\what{z}_1^\ell \geq z_{\min}^\mathfrak{n}$ are defined in \Cref{prop:z_min_defs}.
    \end{lemma}

    \begin{proof}
We verify the bound for $M^\ell$; the proof of the bound on $\what{M}^\ell$ is nearly identical.
Let $M^\ell = (I+J^\ell)^{-1}$, where $J^\ell = \frac{\dt}{\veps}V^\top Z^\ell V$ and $z_1^\ell > 0$ is the minimum positive eigenvalue of $Z^\ell$.
Then $J^\ell$ is symmetric, and for any $\bfx \in \bbR^{N-1}$,
    \begin{equation}
    \label{eq:M_bounds}
\bfx^\top J^\ell\bfx
  \geq
    \frac{\dt}{\veps}z_{1}^\ell\|\bfx\|_2^2
\qquand
\|J^\ell\bfx\|^2_2
  =
    \bfx^\top (J^\ell)^2\bfx
  \geq
    \left(\frac{{\dt}}{\veps}z_{1}^\ell \right)^2\|\bfx\|_2^2
.
    \end{equation}
Given $\bfy \in \bbR^{N-1}$, let $\bfx =  M^\ell \bfy =  (I+J^\ell)^{-1}\bfy$.
Then $\bfy = (I+J^\ell)\bfx$ and, using \eqref{eq:M_bounds},
    \begin{equation}
\|\bfy\|^2_2 
  =
    \|\bfx\|^2_2 + 2 \bfx^\top (J^\ell)\bfx +\|J^\ell\bfx\|^2_2
  \geq
    \|\bfx\|^2_2
    +
    2\frac{\dt}{\veps}z_{1}^\ell \|\bfx\|^2_2
    +
    \left(\frac{{\dt}}{\veps}z_{1}^\ell \right)^2 \|\bfx\|^2_2 
  =
    \left( 1+\frac{\dt}{\veps}z_{1}^\ell\right)^2\|\bfx\|^2_2
.
    \end{equation}
Since $\bfy$ is arbitrary and $z_{\min}^\mathfrak{n}\leq z_1^\ell$, the bounds on $M^\ell$ follow.
    \end{proof}

    \section{\texorpdfstring{Bounding Differences of $Z^\ell$ and $\what{Z}^\ell$ Matrices}{Z(ell),ZHat(ell)}}
    \label{appendix:Z_and_ZHat_diffs}
The following lemma is a consequence of the difference of squares formula.
    \begin{lemma}
    \label{lemma:sqrtIsLipschitz}
Suppose that $x,y,\kappa\in\bbR$ such that $0<\kappa\leq x,y$.
Then
    \begin{equation}
\left|
  \sqrt{x}-\sqrt{y}
\right|
  \leq 
    \frac{1}{2\sqrt{\kappa}}
    |x-y|
.
    \end{equation}
    \end{lemma}

    \begin{lemma}[Bounding 
    \texorpdfstring{$|\Delm{A}|$ and $|\Delm{B}|$ }{|{A,B}(ell)-{A,B}(ell-1)|}]
    \label{lemma:A_ij_B_ij_Diffs}
Suppose $T_i^{\ell+1} \geq \min_k\{T_k^{\mathfrak{n}}\}\eqqcolon T_{\min}^{\mathfrak{n}}$.
Then
    \begin{subequations}
    \begin{align}
    \label{eq:DelA_bound}
|\Delm{(A_{i,j}})|
  &\leq
    C^{\mathfrak{n}}_A
    \left(
      |\Delm{(T_i)}| + |\Delm{(T_j)}| 
    \right)
,
    \\
    \label{eq:DelB_bound}
|\Delm{(B_{i,j})}|
  &\leq
    C_B^\mathfrak{n}
    \left(
      |\Delm{(T_i)}| + |\Delm{(T_j)}| 
    \right)
,
    \end{align} 
    \end{subequations}
where
    \begin{equation}
C_A^{\mathfrak{n}}
  = 
    \frac
        {\max_{i,j}\{c_{i,j}^A\} \; m_{\max}^{\sfrac{1}{2}}}
        {2\sqrt{2}m_{\min}\sqrt{T_{\min}^{\mathfrak{n}}}} 
\qquand
C_B^\mathfrak{n}
  = 
    \frac
        {\max_{i,j}\{c_{i,j}^B\}\;m_{\max}^{\sfrac{1}{2}}}
        {2\sqrt{2}m_{\min}\sqrt{T_{\min}^{\mathfrak{n}}}}
.
    \end{equation}
    \end{lemma}

    \begin{proof}
We prove only the bound in \eqref{eq:DelA_bound}.
The proof of \eqref{eq:DelB_bound} is nearly identical.
Let $\Psi^\ell_{i,j} = \Psi_{i,j}(T_i^\ell,T_j^\ell)$ where $\Psi_{i,j}$ is given in \eqref{eq:Psi}.
Then due to the lower bound on the temperature, 
    \begin{equation}
x \coloneqq (\Psi^\ell_{i,j})^2
    ,\qquad 
y \coloneqq (\Psi^{\ell-1}_{i,j})^2
    ,\quand
\kappa \coloneqq \frac{2T_{\min}^{\mathfrak{n}}}{m_{\max}}
    \end{equation}
satisfy the conditions of \Cref{lemma:sqrtIsLipschitz}.
Therefore,
    \begin{equation}
\left|
  \Delm{(\Psi_{i,j})}
\right|
  \leq
    \frac{1}{2} 
    \sqrt{\frac{m_{\max}}{2T_{\min}^{\mathfrak{n}}}}\;
    \left|
      \frac{\Delm{(T_i)}}{m_i} + \frac{\Delm{(T_j)}}{m_j}
    \right|
,
    \end{equation}
and hence, using \Cref{definition:A_ij_B_ij},
    \begin{equation}
|\Delm{(A_{i,j})}|
  =
    c_{i,j}^A|\Delm{(\Psi_{i,j})}|
  \leq
    C^\mathfrak{n}_A
    \left(
      |\Delm{(T_i)}| + |\Delm{(T_j)}|
    \right)
.
    \end{equation}
    \end{proof}
    
    \begin{prop}[Bounding 
    \texorpdfstring{$\|\Delm{Z}\|_2$ and $\|\Delm{\what{Z}}\|_2$}{|Z(ell)-Z(ell-1)|}]
Suppose that $T_i^{\ell+1} \geq \min_k\{T_k^{\mathfrak{n}}\}\eqqcolon T_{\min}^{\mathfrak{n}}$.
Then
    \begin{equation}
\|\Delm{Z}\|_2
  \leq C_{Z}^\mathfrak{n}
    \|\Delm{\bfeta}\|_2
\qquand
\|\Delm{\what{Z}}\|_2
  \leq C_{\what{Z}}^\mathfrak{n}
    \|\Delm{\bfeta}\|_2
,
    \end{equation}
where
    \begin{equation}
C_Z^\mathfrak{n}
  =
    \frac
        {2 (N + N^{\sfrac{1}{2}}) C_A^\mathfrak{n}}
        {\rho_{\min}n_{\min}^{\sfrac{1}{2}}}
\qquand
C_{\what{Z}}^\mathfrak{n}
    =
      \frac
          {2(N+N^{\sfrac{1}{2}})C_B^\mathfrak{n}}
          {n_{\min}^{\sfrac{3}{2}}}
.
    \end{equation}
    \end{prop}
    
    \begin{proof}
We prove the first inequality; the second is similar.
Since $Z = P^{-\sfrac12}(D-A)P^{-\sfrac12}$ is positive semi-definite, then for $\bfy=P^{-\sfrac{1}{2}}\bfx$,
    \begin{equation}
0 
  \leq
    \bfx^\top \Delm{Z}\bfx
  =
    \bfy^\top \Delm{(D-A)}\bfy
\qquad\Longrightarrow\qquad
\frac{\bfy^\top\Delm{A}\bfy}{\bfy^\top\bfy}
  \leq 
    \frac{\bfy^\top\Delm{D}\bfy}{\bfy^\top\bfy}
.
    \end{equation}
Thus $\|\Delm{A}\|_2 \leq \|\Delm{D}\|_2$, and using \Cref{lemma:A_ij_B_ij_Diffs} in the second line below gives, 
    \begin{equation}
    \begin{split}
&\|\Delm{Z}\|_2 
  \leq
    \frac{1}{\rho_{\min}}\|\Delm{(D-A)} \|_2 
  \leq
    \frac{2}{\rho_{\min}}\|\Delm{D}\|_2
    \\
  &\quad=
    \frac{2}{\rho_{\min}}\max_i
    \big|
      \sum_k\Delm{(A_{i,k})}
    \big|
  \leq
    \frac{2C_A^{\mathfrak{n}}}{\rho_{\min}}\max_i\sum_k(|\Delm{(T_i)}| + |\Delm{(T_k)}|)
    \\
  &\quad\leq
    \frac{2C_A^{\mathfrak{n}}}{\rho_{\min}}N\|\Delm{\bfT}\|_2
    +
    \frac{2C_A^{\mathfrak{n}}}{\rho_{\min}}N^{\sfrac{1}{2}}\|\Delm{\bfT}\|_2
  \leq
    \frac{2C_A^{\mathfrak{n}}}{\rho_{\min}}(N+N^{\sfrac{1}{2}})\frac{\|\Delm{\bfeta}\|_2}{n_{\min}^{\sfrac{1}{2}}}
,
    \end{split}
    \end{equation}
where the last line follow from the relation $\bfeta = Q^{\sfrac12} \bfT$ with $Q_{j,k} = n_k \delta_{j,k}$.
    \end{proof}

    \section{Invariance of \texorpdfstring{$\alpha_{i,j}$}{alpha} and \texorpdfstring{$\beta_{i,j}$}{beta}}

    \begin{lemma}
    \label{lemma:alpha_beta_diffs}
For any $i,j\in\{1,\cdots,N\}$, let
    \begin{equation}
\alpha_{i,j}^\ell
  =
    \frac
        {\rho_i\lambda_{i,j}^\ell}
        {\rho_i\lambda_{i,j}^\ell+\rho_j\lambda_{j,i}^\ell}
\qquand
\beta_{i,j}^\ell
  =
    \frac
        {n_i\lambda_{i,j}^\ell}
        {n_i\lambda_{i,j}^\ell+n_j\lambda_{j,i}^\ell}
,
    \end{equation}
where (cf. \eqref{eq:collisionFrequencies}) $\lambda^\ell_{i,j}=c_{i,j}n_j\Psi^\ell_{i,j}$.
Then 
    \begin{equation}
\alpha_{i,j}^{\ell+1}-\alpha_{i,j}^\ell
    = 0 =
\beta_{i,j}^{\ell+1}-\beta_{i,j}^\ell
.
    \end{equation}
    \end{lemma}
    
    \begin{proof}
Only the first equality is shown.
The other case is similar.
Direct calculation gives
    \begin{equation}
    \label{eq:alpha_diff}
\alpha_{i,j}^{\ell+1}-\alpha_{i,j}^\ell
  =
    \frac
        {
          \rho_i\rho_j
          \left(
            \lambda_{i,j}^{\ell+1}\lambda_{j,i}^{\ell} 
            -
            \lambda_{i,j}^\ell\lambda_{j,i}^{\ell+1}
          \right)
        }
        {
          (\rho_i\lambda_{i,j}^{\ell+1}+\rho_j\lambda_{j,i}^{\ell+1})
          (\rho_i\lambda_{i,j}^{\ell}+\rho_j\lambda_{j,i}^{\ell})
        }
,
    \end{equation}
It is easy to verify that the product $\lambda_{i,j}^{\ell+1}\lambda_{j,i}^{\ell}$ is symmetric in $i,j$.
Hence the difference in \eqref{eq:alpha_diff} is identically zero.
    \end{proof}

    \section{Temperature Source Terms Are Positive}
This appendix contains three technical lemmas used in the proof of \eqref{eq:tempBound} in \Cref{prop:GST_properties}.
    \begin{lemma}
    \label{lemma:s_i_diffs}
    \begin{equation}
s_i^{n}-s_i^{\ell+1}
  =
    \left|
      \bfu_i^{\ell+1}-\bfu_i^{\mathfrak{n}}
    \right|^2
    +
    \frac{\dt}{\veps}\sum_j\lambda_{i,j}^{\ell}\alpha_{j,i}
    \left[
      2s_i^{\ell+1}
      -
      2
      \left(
        \bfu_i^{\ell+1},\bfu_j^{\ell+1}
      \right)_2
    \right]
.
    \end{equation}
    \end{lemma}

    \begin{proof}
Rearrange the iterative step from \eqref{eq:GSTstep} to obtain
    \begin{subequations}
    \begin{align}
\bfu_i^{n}-\bfu_i^{\ell+1}
  &=
    \frac{\dt}{\veps}\sum_j\lambda_{i,j}^{\ell}\alpha_{j,i}(\bfu_i^{\ell+1}-\bfu_j^{\ell+1})
    \\
\text{and}\qquad
\bfu_i^{n} + \bfu_i^{\ell+1}
  &=
    2\bfu_i^{\ell+1} + \frac{\dt}{\veps}\sum_j\lambda_{i,j}^{\ell}\alpha_{j,i}(\bfu_i^{\ell+1}-\bfu_j^{\ell+1})
.
    \end{align}    
    \end{subequations}
Take the inner product of these to obtain the result.
    \end{proof}

    \begin{lemma}
    \label{lemma:technical_identity}
    \begin{subequations}
    \begin{align}
    \label{eq:ident1}
&m_i\beta_{i,j}
  (s_i^{\ell+1}-S_{i,j}^{\ell+1})
  +
m_j\beta_{j,i}
  (s_j^{\ell+1}-S_{i,j}^{\ell+1})
  =
    m_i\alpha_{j,i}\beta_{i,j}
    \left|
      \bfu_i^{\ell+1}-\bfu_j^{\ell+1}
    \right|^2
,
    \\
    \label{eq:ident2}
&\frac{\dt}{\veps d}\sum_j
  \lambda_{i,j}^{\ell}\beta_{j,i}
    \left[m_j
      \left(
        s_j^{\ell+1}-S_{i,j}^{\ell+1}
      \right)
      -
      m_i
      \left(
        s_i^{\ell+1}-S_{i,j}^{\ell+1}
      \right)
    \right]
    \\
    \nonumber
  &\qquad= \frac{\dt}{\veps d}\sum_jm_i\lambda_{i,j}^{\ell}\alpha_{j,i}
  \bigg[
    \beta_{i,j}\left|\bfu_i^{\ell+1}-\bfu_j^{\ell+1}\right|^2
    - (1+\alpha_{i,j})s_i^{\ell+1}
    + \alpha_{j,i}s_j^{\ell+1}
    + 2\alpha_{i,j}\left(\bfu_i^{\ell+1},\bfu_j^{\ell+1}\right)_2
  \bigg]
.
    \end{align}
    \end{subequations}
    \end{lemma}

    \begin{proof}
The proof of \eqref{eq:ident1} is a calculation using the fact that $\alpha_{i,j}+\alpha_{j,i}=1=\beta_{i,j}+\beta_{j,i}$, and the identities
    \begin{equation}
m_j\alpha_{i,j}\beta_{j,i}
  =
    m_i\alpha_{j,i}\beta_{i,j}
\qquiff
\alpha_{j,i}
  (m_i\beta_{i,j} + m_j\beta_{j,i})
  = m_j\beta_{j,i}
.
    \end{equation}
To prove \eqref{eq:ident2}, use \eqref{eq:ident1}.
    \end{proof}

    \begin{prop}
    \label{lemma:TempSourcesPositive}
    \begin{equation}
\frac{m_i}{d}(s_i^{\mathfrak{n}}-s_i^{\ell+1})
  +
    \frac{\dt}{\veps d}\sum_j\lambda_{i,j}^{\ell}\beta_{j,i}
      \left[
        m_j(s_j^{\ell+1}-S_{i,j}^{\ell+1})
          -
        m_i(s_i^{\ell+1}-S_{i,j}^{\ell+1})
      \right]
  \geq 0
.
    \end{equation}
    \end{prop}

    \begin{proof}
Using \Cref{lemma:s_i_diffs,lemma:technical_identity}, and $\alpha_{i,j}+\alpha_{j,i}=1$,
    \begin{equation}
    \begin{split}
\frac{m_i}{d}
  &(s_i^{\mathfrak{n}}-s_i^{\ell+1})
  +
    \frac{\dt}{\veps d}\sum_j\lambda_{i,j}^{\ell}\beta_{j,i}
      \left[
        m_j(s_j^{\ell+1}-S_{i,j}^{\ell+1})
          -
        m_i(s_i^{\ell+1}-S_{i,j}^{\ell+1})
      \right]
    \\
  &= \frac{m_i}{d}
    \left|
      \bfu_i^{\ell+1}-\bfu_i^\mathfrak{n}
    \right|^2
    +
    \frac{\dt}{\veps}\frac{m_i}{d}
    \sum_j\lambda_{i,j}^\ell\alpha_{j,i}
    (\alpha_{j,i}+\beta_{i,j})
    \left|
      \bfu_i^{\ell+1}-\bfu_j^{\ell+1}
    \right|^2
  \geq 0
.
    \end{split}
    \end{equation}
    \end{proof}

    \bibliographystyle{plain}
    \bibliography{references}

    \end{document}